\documentclass[12pt]{amsart}
\usepackage{amssymb}
\usepackage{stmaryrd}
\usepackage{fancyhdr}
\usepackage{amscd,latexsym,amsthm,amsfonts,amssymb,amsmath,amsxtra}
\usepackage[mathscr]{eucal}
\usepackage{geometry}
\usepackage[colorlinks=true,urlcolor=blue,bookmarks=true,bookmarksopen=true,citecolor=blue]{hyperref}
\usepackage[all]{xy}
    \newcommand{\BA}{{\mathbb {A}}} 
    \newcommand{\BC}{{\mathbb {C}}} 
     
    \newcommand{\BG}{{\mathbb {G}}}

    \newcommand{\BQ}{{\mathbb {Q}}}

     \newcommand{\BZ}{{\mathbb {Z}}}

    \newcommand{\CA}{{\mathcal {A}}} 
     \renewcommand{\CD}{{\mathcal {D}}}
    \newcommand{\CE}{{\mathcal {E}}} \newcommand{\CF}{{\mathcal {F}}}

    \newcommand{\CO}{{\mathcal {O}}}

    \newcommand{\RG}{{\mathrm {G}}}

    \newcommand{\fu}{{\mathfrak{u}}}

     \newcommand{\GL}{{\mathrm{GL}}}

    \newcommand{\Ind}{{\mathrm{Ind}}}\newcommand{\ind}{{\mathrm{ind}}}

    \newcommand{\PGL}{{\mathrm{PGL}}}

    \renewcommand{\Re}{{\mathrm{Re}}}

    \newcommand{\Res}{{\mathrm{Res}}}

    \newcommand{\SO}{{\mathrm{SO}}}\newcommand{\Sp}{{\mathrm{Sp}}}
    
    \newcommand{\Sym}{{\mathrm{Sym}}}
    \newcommand{\Stab}{{\mathrm{Stab}}}

    \newcommand{\Span}{{\mathrm{Span}}}
     
    \newcommand{\diag}{\mathrm{diag}} 
    \newcommand{\Mat}{\mathrm{Mat}}

     \newcommand{\wt}{\widetilde}

    \newcommand{\pair}[1]{\langle {#1} \rangle}

    \newcommand{\lra}{\longrightarrow}
     
    \newcommand{\bs}{\backslash}

    \theoremstyle{plain}
    \newtheorem{thm}{Theorem}[section]  \newtheorem{ass}[thm]{Assumption} 
    \newtheorem{lem}[thm]{Lemma}  \newtheorem{prop}[thm]{Proposition} \newtheorem{rmk}[thm]{Remark}
      \newtheorem{defn-prop}[thm]{Definition-Proposition}

    \numberwithin{equation}{section}
\begin{document}

 \title[reciprocal branching problem and Vogan packets]
 {A reciprocal branching problem for automorphic representations and global Vogan packets}


\author{Dihua Jiang}
\address{School of Mathematics, University of Minnesota, Minneapolis, MN 55455, USA}
\email{dhjiang@math.umn.edu}

\author{Baiying Liu}
\address{Department of Mathematics, Purdue University, 150 N. University St, West Lafayette, IN, 47907, USA}
\email{liu2053@purdue.edu}

\author{Bin Xu}
\address{School of Mathematics, Sichuan University, No. 29 Wangjiang Road, Chengdu 610064, P. R. China}
\email{binxu@scu.edu.cn}

\begin{abstract}
Let $G$ be a group and $H$ be a subgroup of $G$. The classical branching rule (or symmetry breaking) asks: For an irreducible representation $\pi$ of $G$,
determine the occurrence of an irreducible representation $\sigma$ of $H$ in the restriction of $\pi$ to $H$. The reciprocal branching problem of this classical branching problem
is to ask: For an irreducible representation $\sigma$ of $H$, find an irreducible representation $\pi$ of $G$ such that $\sigma$ occurs in the restriction
of $\pi$ to $H$. For automorphic representations of classical  groups, the branching problem has been addressed by the well-known global Gan-Gross-Prasad conjecture. 
In this paper, we investigate the reciprocal branching problem for automorphic representations 
of special orthogonal groups using the twisted automorphic descent method as developed in \cite{JZ15}. The method may be applied to 
other classical groups as well. 
\end{abstract}

\thanks{The first named author is partially supported by NSF grant DMS-1600685.
The second named author is partially supported by NSF grant DMS-1702218 and by  start-up funds from the Department of Mathematics at Purdue University. The third named author is partially supported by NSFC grant No.11501382 and by the Fundamental Research Funds for the Central Universities}

\maketitle


\section{Introduction}\label{intro}

The classical {\sl branching rule} or the so called {\sl symmetry breaking} in representation theory is to ask a question
that can be formulated in an over-simplified way as follows:
Let $G$ be a group and $H$ be a subgroup. For an irreducible representation $\pi$ of $G$, the problem is to ask which
irreducible representation $\sigma$ of $H$ occurs in the restriction of $\pi$ to $H$. A refinement of this classical problem is to ask if an irreducible
representation $\sigma$ of $H$ with extra condition can appear in the decomposition when $\pi$ is restricted to $H$.
Such a classical problem has been successfully studied in many different contents.
For automorphic representations of classical groups, the global Gan-Gross-Prasad conjecture addresses this classical problem in a certain format, which will be described with some details below.

The objective of this paper is to consider the automorphic version of the problem reciprocal to the refined classical branching rule problem.
The reciprocal branching problem could be formulated as follows: For an irreducible representation $\sigma$ of $H$, with a certain extra property, find an irreducible representation $\pi$ (possibly with a certain extra property) of $G$, which contains $H$ as a subgroup,
such that $\sigma$ occurs in the restriction of $\pi$ to $H$. It is clear that without those extra conditions on $\pi$ or $\sigma$, the usual
Frobenius reciprocity law suggests that one may take $\pi$ to be the induced representation from $\sigma$. However, with those extra conditions,
the usual induced representation may not be enough for such refined problems.
This paper is to understand this reciprocal problem in the theory of automorphic
representations of special orthogonal groups in terms of global Vogan packets and with connection to the global Gan-Gross-Prasad conjecture. We believe that the method used in this paper should be applicable to other classical groups as well.

\subsection{The branching problem and the global Gan-Gross-Prasad conjecture}
Before introducing the precise problem we will consider, we recall from \cite{GGP12} the global Gan-Gross-Prasad conjecture for the case closely relevant to the topics discussed in this paper.

Let $F$ be a number field and $\BA=\BA_F$ be the ring of adeles of $F$.
Following \cite{A}, we denote by $G_n^*:=\SO_{2n}^{\epsilon}$ an $F$-quasi-split special even orthogonal group, and denote by $H_m^*:=\SO_{2m+1}^*$
the $F$-split special odd orthogonal group. 
Here $\SO_{2n}^\epsilon$ ($\epsilon\in F^\times/(F^\times)^2$) is the $F$-quasi-split special orthogonal group determined by $(n-1)$ hyperbolic planes and $E_\epsilon=F[X]/(X^2-\epsilon)$, and is $F$-split if $\epsilon\in (F^\times)^2$. 
As in \cite{A} and \cite{GGP12}, we denote by $G_n$ and $H_m$ a pure inner form of $G_n^*$ and
$H_m^*$, respectively. Note that $G_n$ and $G_n^*$ share the same Langlands $L$-group, and so do $H_m$ and $H_m^*$. We recall the global Arthur
parameters from \cite{A}, and consider the generic ones mostly. A generic global Arthur parameter for $G_n^*$ is given as a formal sum
\begin{equation}\label{ap1}
\phi=(\tau_1,1)\boxplus\cdots\boxplus(\tau_r,1),
\end{equation}
where $\tau_i$ with $i=1,2,\cdots,r$ is an irreducible unitary, self-dual cuspidal automorphic representation of $\GL_{a_i}(\BA)$,
and $\tau_i\not\cong\tau_j$
if $i\neq j$. Moreover, each $\tau_i$ is of orthogonal type in the sense that the symmetric square $L$-function $L(s,\tau_i,\Sym^2)$ has a pole at $s=1$.
Note that one must have $2n=\sum_{i=1}^ra_i$. The global Arthur parameter $\phi$ as in \eqref{ap1} can be realized as an irreducible
automorphic representation of $\GL_{2n}(\BA)$. The set of generic global Arthur parameters of $G_n^*$ is denoted by $\wt{\Phi}_2(G_n^*)$. For each
Arthur parameter $\phi\in\wt{\Phi}_2(G_n^*)$, the associated global Arthur packet is denoted by $\wt{\Pi}_\phi(G_n^*)$.
Since $G_n$ shares the Langlands $L$-group with $G_n^*$, an Arthur parameter $\phi\in\wt{\Phi}_2(G_n^*)$ may be regarded as an Arthur parameter of $G_n$.
However, the associated global Arthur packet $\wt{\Pi}_\phi(G_n)$ could be empty. If the global Arthur packet $\wt{\Pi}_\phi(G_n)$ is not empty, we call the
global Arthur parameter $\phi$ of $G_n^*$ is $G_n$-{\sl relevant}. The set of all $G_n$-relevant, generic global Arthur parameters of $G_n^*$ is denoted by
$\wt{\Phi}_2(G_n^*)_{G_n}$. Now the global Vogan packet associated to $\phi\in\wt{\Phi}_2(G_n^*)$ is given by
\begin{equation}\label{gvp1}
\wt{\Pi}_\phi[G_n^*]
:=
\cup_{G_n}\wt{\Pi}_\phi(G_n)\,,
\end{equation}
where $G_n$ runs over all pure inner forms of $G_n^*$ over $F$. Similarly, we have global Arthur parameters $\phi'$ for $H_m^*$ as in \eqref{ap1},
\begin{equation}\label{ap2}
\phi'=(\tau'_1,1)\boxplus\cdots\boxplus(\tau'_{r'},1)\,,
\end{equation}
with only
difference that $\tau'_1,\cdots,\tau'_{r'}$ are now of symplectic type. The set of generic global Arthur parameters of $H_m^*$ is denoted by
$\wt{\Phi}_2(H_m^*)$. Accordingly, we have the global Arthur packet $\wt{\Pi}_{\phi'}(H_m^*)$ and the global Vogan packet $\wt{\Pi}_{\phi'}[H_m^*]$.

Assume that $(G_n^*,H_m^*)$ $(m<n)$ is a relevant pair in the sense of the Gan-Gross-Prasad conjecture in \cite{GGP12}, and $G_n\times H_m$ is a relevant
pure inner form of $G_n^*\times H_m^*$ over $F$. The global Vogan packet for $\phi\times\phi'$ is given by
\begin{equation}\label{gvp2}
\wt{\Pi}_{\phi\times\phi'}[G_n^*\times H_m^*]
:=
\cup_{G_n\times H_m}
\wt{\Pi}_{\phi\times\phi'}(G_n\times H_m)\,,
\end{equation}
where $G_n\times H_m$ runs through all relevant pure inner forms of $G_n^*\times H_m^*$ over $F$.
The tensor product $L$-function associated to the pair $\phi$ and $\phi'$ is defined to be
\begin{equation}\label{Lfn}
L(s,\phi\times\phi')
:=
\prod_{i=1}^r\prod_{j=1}^{r'}L(s,\tau_i\times\tau'_j)\,.
\end{equation}
The global Gan-Gross-Prasad conjecture (or GGP conjecture for short) (\cite{GGP12}) asserts that the central value $L(\frac{1}{2},\phi\times\phi')$ is non-zero if and only if
there exists a pair $(\pi_0,\sigma_0)$ in the global Vogan packet $\wt{\Pi}_{\phi\times\phi'}[G_n^*\times H_m^*]$ with a non-zero Bessel period for
$(\pi_0,\sigma_0)$. The uniqueness of such a pair $(\pi_0,\sigma_0)$ follows from the local GGP conjecture (\cite{GGP12}). When
such a pair exists, we call it a {\it the Gan-Gross-Prasad pair} or {\sl GGP pair} for short.

The {\sl branching problem} for this case is to ask: For any 
$\pi$ in the global Vogan packet 
$\wt{\Pi}_{\phi}[G_n^*]$, is there any $\sigma$ in the global 
Vogan packet 
$\wt{\Pi}_{\phi'}[H_m^*]$ such that 
$\pi\otimes\sigma$ belongs to the global Vogan packet 
$\wt{\Pi}_{\phi\times\phi'}[G_n^*\times H_m^*]$ and 
$(\pi,\sigma)$ forms a GGP pair, i.e. $(\pi,\sigma)$ has a 
non-zero Bessel period. The global Gan-Gross-Prasad conjecture 
is to characterize the occurrence of $\sigma$ in 
the branching decomposition of $\pi$ with respect to relevant pair $(G_n, H_m)$ in terms of the central value of the $L$-function,
$L(\frac{1}{2},\phi\times\phi')$ and the local symplectic root 
numbers associated to the pair $(\pi,\sigma)$. 
Because of the uniqueness of the pair $(\pi_0,\sigma_0)$, it makes sense to ask the following question: 

{\bf Reciprocal Branching Problem:}\ 
{\sl For $\sigma\in\wt{\Pi}_{\phi'}(H_m)$, with
$\sigma\not\cong\sigma_0$, how to find some group $G'$ and an irreducible cuspidal automorphic representation $\pi'$ of $G'(\BA)$ with a generic global Arthur parameter $\phi'$, such that
$G'$ and $H_m$ form a relevant pair and $\pi'$ and $\sigma$ have a non-zero Bessel period? }

\subsection{The reciprocal branching problem and the twisted automorphic descent}\label{subsection: intro of reciprocal problem}
We are going to study this reciprocal branching problem for 
automorphic representations of orthogonal groups within the general framework of global Vogan packets and the global Gan-Gross-Prasad conjecture, by means of the twisted automorphic descents.

Assume that $\sigma$ belongs to the global Arthur packet $\wt{\Pi}_{\phi'}(H_m)$ for some pure inner form $H_m$ of $H_m^*$ over $F$.
To present an answer to the reciprocal branching problem in this situation, one may take a generic global Arthur parameter $\phi\in \wt {\Phi}[G_n^*]$ for some quasi-split form $G_n^*$ at first. 
By the implication of the global GGP conjecture, we may also assume $L(\frac{1}{2},\phi\times\phi')\neq 0$.
Assume that $(\pi_0,\sigma_0)$ is the unique GGP pair in the global Vogan packet $\wt{\Pi}_{\phi\times\phi'}[G_n^*\times H_m^*]$, which has a non-zero Bessel period,
as given by the global GGP conjecture.
If $\sigma\simeq \sigma_0$, then the global GGP conjecture predicts that the member $\pi_0$ in the global Vogan packet $\wt\Pi_{\phi}[G_n^*]$ gives an answer to the reciprocal branching problem. The twisted automorphic descent of \cite{JZ15} provides an explicit construction of this $\pi_0$ in terms of the generic global Arthur parameter $\phi$ and $\sigma$. 

In this paper we consider the reciprocal branching problem for any $\sigma$ belonging to the global Vogan packet $\wt{\Pi}_{\phi'}[H_m^*]$ but not equivalent to $\sigma_0$. 
The interesting part of this situation is that, by the uniqueness property in the GGP conjecture, for the fixed parameter $\phi$ of dimension $2n$, one does not expect that there exist $\pi\in \wt{\Pi}_{\phi}[G_n^*]$ such that $\pi$ and $\sigma$ have a non-zero Bessel period. 
The idea is to find an (specific) even special orthogonal group $G_{n+k}$ for some $k\geq 1$, which is a pure inner form of some quasi-split form $G_{n+k}^*$ and is relevant to $H_m$,
and construct explicitly an irreducible cuspidal automorphic representation $\pi_{n+k}$ of $G_{n+k}(\BA)$ with the properties that $\pi_{n+k}$ has
a generic global Arthur parameter and has a non-zero Bessel period with the given $\sigma$. We expect that the integer $k$ should be determined by the
{\sl first occurrence index} of $\sigma$ (see \cite{JZ15} or \S \ref{subsection: construction of automorphic descent} for the definition) in the tower of the Bessel descents from $\sigma$. 
Hence the construction of $G_{n+k}$ and $\pi_{n+k}$ should depend on the structure of the Bessel-Fourier coefficients of $\sigma$, which reflects the natural relation between the Bessel-Fourier coefficients and the twisted automorphic descents. 

We will give an answer to this reciprocal branching problem when $m=1$.
We fix a non-trivial additive character $\psi: F\bs \BA\lra \BC^\times$ and fix an positive integer $n$.
Let 
\begin{equation}\label{expression of tau}
\tau=\tau_1\boxplus\tau_2\boxplus\cdots\boxplus\tau_{r} 
\end{equation}
be the isobaric sum automorphic representation of $\GL_{2n}(\BA)$.
Here $\tau_i$ is a unitary irreducible cuspidal automorphic representation of $\GL_{n_i}(\BA)$ such that $\sum_{i=1}^r n_i=2n$.
Assume that $\tau$ corresponds to the generic Arthur parameter $\phi_\tau=\phi_{\tau_1}\boxplus\cdots\boxplus \phi_{\tau_r}$. Moreover, We assume that each $\tau_i$ is of orthogonal type, i.e. the $L$-function $L(s,\tau_i,\Sym^2)$ has a pole at $s=1$. Hence $\phi_\tau$ is a generic global Arthur parameter of some $G_n^*$.

On the other hand, we assume that $\tau_0$ is an irreducible unitary cuspidal representation of $\GL_2(\BA)$ of symplectic type and has the property that
$L(\frac{1}{2}, \tau\times \tau_0)\neq 0$.
Let $V_0$ be a quadratic space over $F$ of dimension $3$ and $H_1^{V_0}=\SO(V_0)$ be the corresponding special orthogonal group. In any case,
the $F$-split group is $H_1^*=\SO_3$. We also denote by $J_{V_0}$ the quadratic form of $V_0$.
Let $\sigma$ be an irreducible cuspidal representation of $H_1^{V_0}(\BA)$ parametrized by $\phi_{\tau_0}$.
For simplicity, we say a quasi-split form $G_n^*$ is relevant to $H_1^{V_0}$ if one of its pure inner form $G_n$ is relevant to $H_1^{V_0}$.
In this paper, we assume
\begin{ass}\label{depth assumption}
The representation $\sigma$ does not occur in the Gan-Gross-Prasad pair $(\pi_0,\sigma_0)$ in the global Vogan $L$-packet
$\wt\Pi_{\phi_\tau\times \phi_{\tau_0}}[G^*_n\times H_1^*]$
for the fixed positive integer $n$ and any quasi-split form $G_n^*$ relevant to $H_1^{V_0}$.
\end{ass}

The goal of this paper is, under Assumption \ref{depth assumption},
to construct an even special orthogonal group $G_{n+1}$ such that $G_{n+1}\times H^{V_0}_1$ is a relevant pure inner form of
$G_{n+1}^*\times H_1^*$, and construct an irreducible cuspidal automorphic representation $\pi_{n+1}$ of $G_{n+1}(\BA)$ such that
$\pi_{n+1}$ has a generic Arthur parameter and has a non-zero Bessel period with respect to $\sigma$. The more precise explanation is in order.

For the given representations $(\tau,\sigma)$, we construct in \S \ref{section of residual representation} a square-integrable residual automorphic
representation $\CE_{\tau\otimes\sigma}$ on $\SO^{V_0}_{4n+3}(\BA)$ with support $(P,\tau\otimes\sigma)$, where $P=MN$ is
the standard parabolic subgroup of $\SO^{V_0}_{4n+3}$ such that
$
M\simeq \GL_{2n}\times H_1^{V_0}.
$
A tower of {\it automorphic descent of $\tau$ twisted by $\sigma$}, which we denoted by $\pi_{\ell,\beta}=\CD_{\psi_{\ell,\beta}}(\CE_{\tau\otimes\sigma})$ with $\beta\in F^\times$, can be constructed by taking the Bessel-Fourier coefficients of certain depth $\ell$ (see \S \ref{section of residual representation}) of the residual representation $\CE_{\tau\otimes\sigma}$. For each $\ell$, if the twisted descent $\pi_{\ell,\beta}=\CD_{\psi_{\ell,\beta}}(\CE_{\tau\otimes\sigma})$ is non-zero, then it consists of certain automorphic functions on $G_{2n-\ell+1,\beta}(\BA)=\SO_{4n-2\ell+2,\beta}(\BA)$, with moderate growth.
After choosing a suitable basis, the group $G_{2n-\ell+1,\beta}$ can be arranged to associate with the symmetric matrix
$$
\wt J=\begin{pmatrix}
   &&w_{2n-\ell-1}\\&J_{V_{0,\beta}}&\\w_{2n-\ell-1}&&
  \end{pmatrix},
$$
where $w_i$ is an ($i\times i$)-matrix with only $1$'s on its anti-diagonal.
Note that it could be either $F$-split, quasi-split or non-quasi-split, where $J_{V_{0,\beta}}=\begin{pmatrix}
J_{V_0}&\\&-\beta
\end{pmatrix}$ is a $(4\times4)$-symmetric matrix that 
defines the $4$-dimensional quadratic space $V_{0,\beta}$.
We also denote by $\eta_{V_{0,\beta}}: F^\times \bs \BA^\times\lra \{\pm 1\}$ the quadratic character associated to the quadratic space $V_{0,\beta}$.
We usually call $G_{2n-\ell+1,\beta}$ the {\it target group} of the descent module $\CD_{\psi_\ell,\beta}(\CE_{\tau\otimes\sigma})$.
The point here is that under the Assumption \ref{depth assumption}, the first occurrence in this tower is at the depth $\ell^*=n$. In this case, we
denote the resulting representation by $\pi_{\beta}=\CD_{\psi_n,\beta}(\CE_{\tau\otimes\sigma})$, which consists of certain automorphic functions on $G_{n+1,\beta}(\BA)=\SO_{2n+2,\beta}(\BA)$, with moderate growth. The main results of this paper can be summarized as follows.

\begin{thm}\label{main thm for descent}
Let $\tau$ and $\sigma$ be as given above.
\begin{itemize}
\item[(a)] There exists $\beta\in F^\times$ such that the twisted descent $\pi_\beta=\CD_{\psi_n,\beta}(\CE_{\tau\otimes\sigma})\neq 0$.
\item[(b)] If $\sigma$ satisfies Assumption \ref{depth assumption}, then $\pi_\beta=\bigoplus_{i} \pi_\beta^{(i)}$ is a multiplicity free
direct sum of irreducible cuspidal automorphic representations $\pi_\beta^{(i)}$ of $\SO_{2n+2,\beta}(\BA)$.
Moreover, if $\omega_\tau=1$ and the $\beta\in F^\times$ in Part (a) is not a square, or if $\omega_\tau\neq 1$ and the $\beta\in F^\times$ in Part (a) is a square, then the same result holds without the above assumption.
\item[(c)] Assume $\omega_{\tau}=1$, and the $\beta\in F^\times$ in Part (a) is not a square. Then
 each irreducible component $\pi_\beta^{(i)}$ of $\pi_{\beta}$ belongs to a global Arthur packet with a generic global Arthur parameter
$\phi^{(i)}$. The parameter $\phi^{(i)}$ has the central character $\eta_{V_{0,\beta}}$, and has the property that
$L(\frac{1}{2}, \phi^{(i)}\times\phi_{\tau_0})\neq 0$ and $(\pi_{\beta}^{(i)}, \sigma)$ has a non-zero Bessel period.
\end{itemize}
\end{thm}
Theorem \ref{main thm for descent} gives an answer to the reciprocal branching problem for $\sigma$, and one may refer to Theorem \ref{main thm for reciprocal problem} at the end of this paper.

Part (a) of Theorem \ref{main thm for descent} provides a base for establishing the main results in this paper and will be proved in \S \ref{non-vanishing}.
It is always the hardest part in the theory of global automorphc descents to prove the global non-vanishing of the construction. In this paper, we 
find a new argument to do so. It is a combination of two methods recently developed. 
One of them is the result in \cite{GGS17} on relations between degenerate Whittaker models and generalized Whittaker models of representations (see \S \ref{FCdefinition} for definitions). Another one is the result in \cite{JLS16} on raising nilpotent orbits in the wave front set of representations. 
They are indispensable for this new way to establish the global non-vanishing of the construction of the twisted automorphic descents. 
More explicitly, first, using results in \cite{GGS17}, we show that $\CE_{\tau \otimes \sigma}$ has a non-zero generalized Whittaker-Fourier coefficient attached to the partition $[(2n)^2 1^3]$, which is not special.
Then, using results in \cite{JLS16}, we show that $\CE_{\tau \otimes \sigma}$ has a non-zero generalized Whittaker-Fourier
coefficient attached to the partition $[(2n+1)(2n-1)1^3]$, which is the smallest orthogonal special partition bigger than $[(2n)^2 1^3]$.
This eventually implies that there exists $\beta \in F^{\times}$, such that
$\CD_{\psi_{n,\beta}}(\CE_{\tau \otimes \sigma})$ is non-zero.
We note that the method we use here is different from that previously used in the theory of descents (see, for example, \cite{GRS11,JLXZ14}), and
is more conceptual and valid for more general situation (see Remark \ref{remark for non-vanishing in general cases}).

Part (c) of Theorem \ref{main thm for descent} is a connection of the descent construction to the reciprocal branching problem we are considering. 
Briefly the argument is as follows.
Let $\pi_\beta^{(i)}$ be an irreducible component of the descent $\pi_\beta$, which is cuspidal. By the unramified structure and 
the endoscopic classification of Arthur (\cite{A}), $\pi_{\beta}^{(i)}$ has a generic Arthur parameter $\phi^{(i)}$.
By unfolding of Eisenstein series as used in \cite{JZ15}, we can show that the Bessel period for $(\pi_\beta^{(i)},\sigma)$ is non-zero.  
Note that the main result of \cite{JZ15} also tells that the central value
$L(\frac{1}{2}, \phi^{(i)}\times\phi_{\tau_0})$ is also non-zero (see \S \ref{GGP}).


\medskip

The content of this paper is organized as following. In \S \ref{residual rep}, we introduce the construction of the automorphic descent in the case
considered in this paper, and study some of its local structures in \S \ref{local aspects}.
In \S \ref{cuspidality} -- \S \ref{non-vanishing}, we study some global aspects of the descent.
In particular, we show the cuspidality of the descent in \S \ref{cuspidality}, and the non-vanishing of the descent in \S \ref{non-vanishing}.
Hence we obtain  Part (a) and Part (b) of Theorem \ref{main thm for descent}.
Finally in \S \ref{GGP}, we prove Part (c) of Theorem \ref{main thm for descent}.

\subsection*{Acknowledgement}
Parts of this paper were written in the Spring of 2016 when the third named author visited the School of Mathematics, University of Minnesota.
He appreciates very much the hospitality and comfortable working condition provided by the School of Mathematics. We would like to thank Lei Zhang for helpful comments. 

\section{Residual representations and the descent construction}\label{section of residual representation}\label{residual rep}

\subsection{Notation}\label{notation}

We set up some general notation that will be used throughout this paper. Let $F$ be a field of characteristic $0$, and $V$ be a quadratic space of dimension $4n+3$ defined over $F$,  with
quadratic form denoted by $\pair{\ , \ }$. Let $\wt m$ be the Witt index of $V$. Then $V$ has a polar decomposition
$$V=V^+\oplus W \oplus V^-\,,$$
where $V^+$ has dimension $\wt m$ and is a maximal totally isotropic subspace of $V$, and $W$ is anisotropic of dimension $4n+3-2\wt m$.
Fix a maximal flag
$$\CF: \quad \{0\}\subset V_1^+\subset V_2^+\subset \cdots \subset V_{\wt m}^+=V^+$$
in $V^+$, and choose a basis $\{e_1, \cdots, e_{\wt m}\}$ of $V^+$
over $F$ such that $V_i^+=\Span\{e_1,\cdots, e_i\}$, for $1 \leq i \leq \wt{m}$. Let $\{e_{-1},\cdots, e_{-\wt m}\}$ be a basis for $V^-$, which is dual to $\{e_1, \cdots, e_{\wt m}\}$, that is,
$$\pair{e_i,e_{-j}}=\delta_{i,j}, \ \text{for}\ 1\leq i,j\leq \wt m\,.$$

We have two cases to consider: $\wt m=2n$ or $\wt m=2n+1$. If $\wt m=2n$,
take a basis $\{e_0^{(1)}, e_0^{(2)}, e_0^{(3)}\}$ of $W$ such that the anisotropic quadratic form on $W$ is associated with $\displaystyle{J_{W}=\begin{pmatrix}1&&\\&\delta&\\ &&\alpha\end{pmatrix}}$,
where $\delta,\alpha \in F^\times$. In this case, the ternary quadratic form $x^2+\delta y^2+\alpha z^2$ does not represent $0$ in $F$, and 
the form of $V$ is associated with
$J=\begin{pmatrix}
   &&w_{2n}\\
   &J_{W}&\\
   w_{2n}&&
\end{pmatrix}$,
where $w_{i}$ is a $(i \times i)$-matrix with $1$'s on the anti-diagonal and zero's elsewhere.
If $\wt m=2n+1$, take a basis $\{e_0\}$ of $W$ with $\pair{e_0,e_0}=1$. Then the form of $V$ is associated with $J=w_{4n+3}$.
For each case, let $H=\SO(V)$, which could be $F$-split or $F$-non-split, according to the above two cases.

For $1\leq \ell\leq \wt m$, let $Q_{\ell}$ be the (maximal) parabolic subgroup fixing the flag
$$ 0\subset V_{\ell}^+\subset V^+\,.$$
Then $Q_\ell$ has a Levi decomposition $M'_{\ell} \cdot U_{\ell}$ with Levi subgroup $M'_{\ell}\simeq \GL_{\ell} \times \SO(V^{(\ell)})$,
where $V^{(\ell)}$ is the subspace which sits into the decomposition
$$V=V_\ell^+\oplus V^{(\ell)} \oplus V_\ell^-\,.$$
For simplicity, when $\ell=2n$, let $P=Q_{2n}$, $M=M'_{2n}$, $U=U_{2n}$, and
$V_0=V^{(2n)}$. Then $M\simeq \GL_{2n}\times \SO(V_0)$, and  $\SO(V_0)$ is an $F$-
split or $F$-non-split special orthogonal group with $\mathrm{dim}\, V_0 = 3$. Accordingly, we may sometimes denote $H=\SO(V)$ by $\SO_{4n+3}^{V_0}$ if we want to indicate the structure of
the group. Note that if $\wt m=2n$, we have $V_0=W$, the anisotropic kernel of $V$ (and hence $J_{V_0}=J_W$).

For $1\leq \ell \leq \wt m$, we also let $P_\ell$ be the parabolic subgroup of $H$ which stabilizes the partial flag
$$\CF_\ell: \quad 0\subset V_1^+ \subset V_2^+ \subset\cdots \subset V_\ell^+\,.$$
It has a Levi decomposition $P_\ell=M_\ell\cdot N_\ell$ with Levi subgroup $M_\ell\simeq \GL_1^\ell\times \SO(V^{(\ell)})$.

\subsection{The residual representations}\label{subsection: residual rep}
Now we take $F$ to be a number field and denote $\BA=\BA_F$ to be its ring of adeles.
Let $\tau=\tau_1\boxplus\tau_2\boxplus\cdots\boxplus \tau_r$ be an isobaric sum automorphic representation of $\GL_{2n}(\BA)$, and $\sigma$ an irreducible unitary cuspidal representation of $\SO(V_0)(\BA)$.
Here $\tau_i$'s are unitary irreducible cuspidal automorphic representations of $\GL_{n_i}(\BA)$, satisfying $\sum_{i=1}^r n_i=2n$.
Take $H=\SO_{4n+3}^{V_0}$ with $\wt m=2n$ as before.
For $s\in \BC$ and an automorphic function $H(\BA)$
$$\phi_{\tau\otimes \sigma}\in \CA(M(F)U(\BA)\bs H(\BA))_{\tau\otimes \sigma}\,,$$
following \cite[\S II.1]{MW95}, one defines
$\lambda_s \phi_{\tau\otimes \sigma}$ to be
$(\lambda_s\circ m_P)\phi_{\tau\otimes \sigma}$, where $\lambda_s \in X_{M}^H\simeq \BC$ (see \cite[\S I.1]{MW95} for the definition of $X_M^H$ and the map $m_P$),
and defines the corresponding Eisenstein series
$$E(s,h,\phi_{\tau\otimes\sigma})=\sum_{\gamma\in P(F)\bs H(F)}\lambda_s \phi_{\tau\otimes \sigma}(\gamma h)\,,$$
which converges absolutely for $\Re(s)\gg 0$ and has meromorphic continuation to the whole complex plane (\cite[\S IV]{MW95}).
Under the normalization of Shahidi (\cite{Sh90}), one may take $\lambda_s=s\wt\alpha$, where $\alpha$ is the unique reduced root of the maximal $F$-split torus of $H$ in $U$.

When $L(s,\tau_i,\Sym^2)$ has a pole at $s=1$ for all $i=1,\cdots,r$ (i.e. $\tau$ is of orthogonal type), and $L(\frac{1}{2},\tau\times\sigma)\neq 0$, the Eisenstein series $E(s,h,\phi_{\tau\otimes\sigma})$ has a pole at $s=\frac{1}{2}$ of order $r$ (see \cite[Proposition 5.3]{JZ15}).
Let $\CE_{\tau\otimes \sigma}$ denote the automorphic representation of $H(\BA)$ generated by the iterated residues
$\Res_{s=\frac{1}{2}}E(s,h,\phi_{\tau\otimes \sigma})$
for all
$$\phi_{\tau\otimes \sigma}\in \CA(M(F)U(\BA)\bs H(\BA))_{\tau\otimes \sigma}\,.$$
It is square integrable by the $L^2$-criterion in \cite{MW95}. Moreover, the residual representation $\CE_{\tau\otimes \sigma}$ is irreducible (see \cite[Theroem A]{M11}).
Note that the global Arthur parameter (\cite{A}) for $\CE_{\tau\otimes \sigma}$ is
$(\tau_1, 2)\boxplus (\tau_2, 2)\boxplus \cdots \boxplus (\tau_r, 2) \boxplus \phi_\sigma$ (\cite[\S 6]{JZ15}).

\subsection{The twisted automorphic descent}\label{subsection: construction of automorphic descent}

The twisted automorphic descents was introduced in \cite{JZ15}, which extends the automorphic descent of Ginzburg-Rallis-Soudry
(\cite{GRS11}) to much more general situation. Following \cite{JZ15} and \cite{GRS11}, we introduce a family of Bessel-Fourier coefficients that defines the
descent.

For $1\leq\ell<2n$, we consider the parabolic subgroup $P_\ell=M_\ell N_\ell$ of $H$ defined in \S \ref{notation}.
The nilpotent subgroup
$N_{\ell}(F)$ consists of elements of the form
$$u=u_\ell(z,x,y)=
\begin{pmatrix}
z&z\cdot x&y\\
&I_{4n+3-2\ell}&x'\\
&&z^*
\end{pmatrix},
$$
where $z\in Z_{\ell}(F)$, $x\in \Mat_{\ell\times (4n+3-2\ell)}(F)$, and $y\in \Mat_{\ell\times \ell}(F)$. Here $Z_\ell$ is the maximal upper-triangular
unipotent subgroup of $\GL_\ell$, and $x'=J^{(\ell)}\,\!^t\!x w_\ell$, where $J^{(\ell)}=\begin{pmatrix}
   &&w_{2n-\ell}\\
   &J_{V_0}&\\
   w_{2n-\ell}&&
\end{pmatrix}$.
Take an anisotropic vector $w_0\in W_\ell$ with $\pair{w_0,w_0}$ in a given square class of $F^\times$, and define a homomorphism
$\chi_{\ell,w_0}: N_\ell\lra \BG_a$ by
$$\chi_{\ell,w_0}(u)=
\sum_{i=2}^\ell \pair{u\cdot e_i,e_{-(i-1)}}+\pair{u \cdot w_0,e_{-\ell}}\,.$$
Define also a character
\begin{equation}\label{G-G character}
\psi_{\ell,w_0}=\psi_F\circ\chi_{\ell,w_0}: N_\ell(\BA)\lra \BC^\times\,,
\end{equation}
with $\psi_F: F\bs \BA\lra \BC^\times$ being a fixed non-trivial additive character.
It is clear that the character $\psi_{\ell,w_0}$ is trivial on $N_{\ell}(F)$.
Now the adjoint action of $M_\ell$ on $N_\ell$ induces an action of $\SO(V^{(\ell)})$ on the set of all such characters $\psi_{\ell,w_0}$.
The stabilizer $L_{\ell,w_0}$ of $\chi_{\ell, w_0}$ in $\SO(V^{(\ell)})$ is equal to $\SO(w_0^\perp\cap V^{(\ell)})$.

Let $\Pi$ be an automorphic representation of $H(\BA)$. For $f\in V_\Pi$ and $h\in H(\BA)$, we define the
$\psi_{\ell,w_0}$-Fourier coefficients of $f$ by
\begin{equation}\label{Gelfand-Graev}
 f^{\psi_{\ell, w_0}}(h)=\int_{[N_\ell]} f(vh)\psi_{\ell,w_0}^{-1}(v)\, \mathrm{d}v\,,
\end{equation}
where $[N_\ell]$ denotes the quotient $N_\ell(F)\bs N_\ell(\BA)$. This is one of the Fourier coefficients of $f$ associated to the partition
$[(2\ell+1)1^{4n+2-2\ell}]$. As $\ell$ varies, it produces a family of Bessel-Fourier coefficients of $f$ needed in this paper.
It is clear that $f^{\psi_{\ell, w_0}}(h)$ is left $L_{\ell,w_0}(F)$-invariant, and is of moderate growth on the Siegel domain of
$L_{\ell,w_0}(F)\bs L_{\ell,w_0}(\BA)$.

As explained in \cite[\S 2]{JLXZ14} and \cite[\S 2.4]{JZ15}, for $1\leq \ell <2n$, we  may take
$$
w_0=y_\beta=e_{2n}+\frac{\beta}{2}e_{-2n}
$$
for some $\beta\in F^\times$ as a precise choice of $w_0$. We have $\pair{y_{\beta},y_{\beta}}= \beta$, and also denote $\psi_{\ell,\beta}=\psi_{\ell, y_\beta}$ for simplicity.
Then in matrix form, we have
$$\begin{aligned}
\psi_{\ell,\beta}(u_{\ell}(z,x,y))
&=\psi(z_{1,1}+\cdots+z_{\ell-1,\ell}+x_{\ell,2n-\ell}+\frac{\beta}{2}x_{\ell,2n+4-\ell})\\
&=\psi_{Z_\ell}(z)\psi(x_{\ell,2n-\ell}+\frac{\beta}{2}x_{\ell,2n+4-\ell})\,.
\end{aligned}
$$
It follows that the stabilizer of the character is
$L_{\ell,\beta}:=L_{\ell,y_{\beta}}=\SO_{4n+2-2\ell,\beta}$, which is
an even special orthogonal group associated to the symmetric matrix
$$
\wt J=\begin{pmatrix}
   &&w_{2n-\ell-1}\\&J_{V_{0,\beta}}&\\w_{2n-\ell-1}&&
  \end{pmatrix},
$$
where
$\displaystyle{J_{V_{0,\beta}}=
\begin{pmatrix}
 J_{V_0}&\\
 &-\beta
\end{pmatrix}
}$.
This group can be split, quasi-split or non-quasi-split over $F$, depending on the choice of $V_0$ and $\beta$ (Proposition 2.5 of \cite{JZ15}).
We also denote by $V_{0,\beta}$ the $4$-dimensional quadratic space over $F$ associated to $J_{V_{0,\beta}}$.
Define $\CD_{\psi_{\ell,\beta}}(\Pi)$ to be the space of $L_{\ell,\beta}(\BA)$-span of the automorphic functions
$f^{\psi_{\ell, w_0}}\big|_{L_{\ell,\beta}(\BA)}$ with all $f\in V_\Pi$, 
where the group $L_{\ell,\beta}(\BA)$ acts by right translation.

Following \cite{JZ15}, the twisted automorphic descent construction in this paper is to take $\Pi=\CE_{\tau\otimes\sigma}$ in the above construction,
here $\CE_{\tau\otimes\sigma}$ is the residual representation we have introduced before.
In this case, we have a family of automorphic $L_{\ell,\beta}(\BA)$-modules:
$$\pi_{\ell,\beta}=\CD_{\psi_{\ell,\beta}}(\CE_{\tau\otimes\sigma})\,.$$
It has been seen in many previous works (see \cite{GRS11, JLXZ14, JZ15}) that $L_{\ell,\beta}(\BA)$-modules $\pi_{\ell,\beta}$ satisfy the
so-called {\it tower property} when the depth $\ell$ varies. 
That is, there exists an $\ell^*$ such that $\pi_{\ell^*,\beta}\neq 0$ for some choice of data, and $\pi_{\ell,\beta}=0$ for all $\ell^*<\ell<2n$.
We call $\ell^*$ the {\it first occurrence index} (of $\sigma$) for the tower $\{\pi_{\ell,\beta}\}_\ell$. 
In particular, at the first occurrence index $\ell^*=\ell^*(\CE_{\tau\otimes\sigma})$,
the $L_{\ell^*,\beta}(\BA)$-module $\pi_{\ell^*,\beta}=\CD_{\psi_{\ell^*,\beta}}(\CE_{\tau\otimes\sigma})$ consists of cuspidal automorphic functions $f^{\psi_{\ell^*,\beta}}(\cdot)$ (see \cite{GRS11, JLXZ14, JZ15}).
In this case we denote $\pi_\beta=\pi_{\ell^*,\beta}$ for simplicity,
and call it {\it the automorphic descent of $\tau$ to $L_{\ell^*,\beta}(\BA)$, {\it twisted} by $\sigma$} or $\sigma$-twisted automorphic descent of $\tau$.
To be compatible with the notation before, we denote $G_{2n-\ell+1,\beta}=L_{\ell,\beta}$, and call $G_{2n-\ell^*+1,\beta}(\BA)$ the {\it target group} of this descent construction.

In this paper, we will restrict ourselves to the pair $(\tau,\sigma)$ that satisfies Assumption \ref{depth assumption} in \S \ref{intro}.
In this situation, the first occurrence index is $\ell^*=n$ (see \S \ref{cuspidality}--\S \ref{non-vanishing}), and the
target group is $G_{n+1,\beta}(\BA)=\SO_{2n+2,\beta}(\BA)$, as stated in Theorem \ref{main thm for descent}.

\section{Local aspects of the descent}\label{local aspects}

In this section, we study the local analogue of the automorphic descent discussed in Section \ref{subsection: construction of automorphic descent}, which is certain twisted Jacquet module. We will calculate these Jacquet modules at unramified places.

\subsection{The twisted Jacquet modules}
We define the twisted Jacquet modules in both cases we are considering. Let $F$ be a $p$-adic field of characteristic $0$, and fix a non-trivial additive character $\psi: F\lra \BC^\times$.
Let $H=\SO(V)$ be a special orthogonal group over $F$ with $\dim(V)=4n+3$, and suppose that the Witt index $\wt m=2n$ or $2n+1$. Let $1\leq\ell\leq \wt m$ and take an anisotropic vector $w_0\in V^{(\ell)}$. We define the character $\psi_{\ell,w_0}$
on $N_{\ell}(F)$ similar to (\ref{G-G character}) in the global setting. Note that if $\ell=\wt m$, we take $w_0\in W=V^{(\wt m)}$, the anisotropic kernel of $V$ (see \S \ref{notation}).

For an irreducible admissible representation $\Pi$ of $H(F)$, one defines the {\it Jacquet module of Bessel type} (of depth $\ell$) to be
\begin{equation}
J_{\psi_{\ell,w_0}}(\Pi):= \Pi/\Span\{\pi(u)\xi-\psi_{\ell,w_0}(u)\xi\ |\ u\in N_\ell(F), \ \xi\in \Pi \}\,.
\end{equation}
Embed $\SO(V^{(\ell)})$ into $H$ via the Levi subgroup $M_\ell$, and set
$L_{\ell,w_0}=\Stab_{\psi_{\ell,w_0}}(\SO(V^{(\ell)}))$
as before. By definition, $J_{\psi_{\ell,w_0}}(\Pi)$ is a $L_{\ell,w_0}(F)$-module.
As before, when $\ell<\wt m$, we may take $w_0=y_\beta=e_{2n}+\frac{\beta}{2}e_{-2n}\in V^{(\ell)}$ with $\beta\in F^\times$. Then we set $J_{\psi_{\ell,y_{\beta}}}(\Pi)=J_{\psi_{\ell,\beta}}(\Pi)$, and $L_{\ell,y_\beta}=L_{\ell,\beta}$.
With a suitable choice of basis, $L_{\ell,\beta}(F)$ is determined by the symmetric matrix
$$\wt J=
\begin{cases}
\begin{pmatrix}
   &&w_{2n-\ell-1}\\&J_{W,\beta}&\\w_{2n-\ell-1}&&
  \end{pmatrix}, & \text{if $\wt m=2n$}\,,\\
  \\
  \begin{pmatrix}
   &&w_{2n-\ell}\\&J_{\beta}&\\w_{2n-\ell}&&
  \end{pmatrix}, & \text{if $\wt m=2n+1$}\,,
  \end{cases}$$
where
$J_{\beta}=
\begin{pmatrix}
1&\\
&-\beta
\end{pmatrix}$.
Recall that we have $V_0=W$ if $\wt m=2n$ (see \S \ref{notation}).

\subsection{The local unramified calculation of Jacquet modules}

We keep some notation which are used in \cite[Chapter 5]{GRS11}.

\textbf{Case I: trivial central character.}
Let $\tau$ be an irreducible, generic, admissible representation of $\GL_{2n}(F)$, which is of orthogonal type. We assume moreover that $\omega_\tau=1$. To consider the unramified cases, we write $\tau$ as a fully induced representation from the Borel subgroup:
\begin{equation}\label{unramified with trivial central character}
\tau=\mu_1\times\cdots\times \mu_n \times \mu_n^{-1}\times\cdots \times \mu_1^{-1}\,,
\end{equation}
where $\mu_i$'s are unramified characters of $F^\times$.
First we consider the case that $\wt m=2n+1$, and hence $H=\SO_{4n+3}(F)$ is $F$-split.
Let $\sigma=\Ind_{B_{\SO_3}(F)}^{\SO_3(F)} \xi$, here $\xi$ is also an unramified character of $F^\times$. Let $\pi_{\tau\otimes\sigma}$ be the unramified constituent of the induced representation
$\Ind_{P(F)}^{\SO_{4n+3}(F)}(\tau \lvert \cdot \rvert^{1/2}\otimes\sigma)$. We want to study the unramified constituents of the twisted Jacquet module $J_{\psi_{\ell,\beta}}(\pi_{\tau\otimes\sigma})$.

\begin{prop}\label{prop for unramified constituents I.1}
	Assume that $\wt m=2n+1$. Then the following hold:
	\begin{enumerate}
		\item Suppose that $\beta\in F^\times-{(F^\times)}^2$. Then the twisted Jacquet module $J_{\psi_{\ell,\beta}}(\pi_{\tau\otimes\sigma})=0$ for all $n+1\leq \ell\leq 2n$. And when $\ell=n$, any unramified constituent of $J_{\psi_{n,\beta}}(\pi_{\tau\otimes\sigma})$ is a subquotient of
		$\Ind^{\SO_{2n+2,\beta}(F)}_{B_{\SO_{2n+2,\beta}}(F)} \mu_1\otimes\cdots\otimes\mu_n\otimes 1\,.$
		\item Suppose that $\beta \in {(F^\times)}^2$. Then the twisted Jacquet module $J_{\psi_{\ell,\beta}}(\pi_{\tau\otimes\sigma})=0$ for all $n+2 \leq \ell\leq 2n$. When $\ell=n+1$, any unramified constituent of $J_{\psi_{n+1,\beta}}(\pi_{\tau\otimes\sigma})$ is a subquotient of
		$\Ind_{B_{\SO_{2n}}(F)}^{\SO_{2n}(F)} \mu_1\otimes\cdots\otimes\mu_n\,,$
		and when $\ell=n$, the unramified constituent of $J_{\psi_{n,\beta}}(\pi_{\tau\otimes\sigma})$ is a subquotient of
		$$
		\begin{aligned}&\quad \left(\bigoplus_\lambda \Ind_{B_{\SO_{2n+2}}(F)}^{\SO_{2n+2}(F)} \mu_1\otimes\cdots\otimes\mu_n\otimes\lambda\right)\\
		&\oplus\left(\bigoplus_{t=1}^n \Ind_{P^{1,\cdots,1,2,1,\cdots,1}_{\SO_{2n+2}}(F)}^{\SO_{2n+2}(F)} \mu_1\otimes\cdots\otimes \mu_{t-1}\otimes\mu_{t}(\det\!_{\GL_2})\otimes\mu_{t+1}\otimes\cdots\otimes\mu_n\right),
		\end{aligned}
		$$
		where $\lambda$ runs over all unramified representations (characters) of occurring in the restriction of $\sigma$ to the $F$-split torus $\SO_2(F)$.
	\end{enumerate}
\end{prop}

\begin{proof}
	We will apply \cite[Theorem 5.1]{GRS11} to calculate $J_{\psi_{\ell,\beta}}(\pi_{\tau\otimes\sigma})$, and we also follow the same notation there.
	By conjugation of some Weyl element, it suffices to consider the unramified constituent of the induced representation $\Ind_{Q_{2n}(F)}^{\SO_{4n+3}(F)}\tau'\otimes\sigma$ instead of $\pi_{\tau\otimes\sigma}$, here
	$$\tau'=\Ind_{P_{2,\cdots,2}(F)}^{\GL_{2n}(F)}\mu_1(\det\!_{\GL_2})\otimes\cdots\otimes\mu_{n}(\det\!_{\GL_2})\,.$$
	Note that the derivative (see \cite{BZ2} and \cite[\S 5]{GRS11}) $\tau'^{(\ell)}$ of $\tau'$ vanishes for $\ell\geq n+1$. Now applying \cite[Theorem 5.1 (i)]{GRS11} with $j=2n<\wt m$, one can see that
	the corresponding twisted Jacquet module vanishes for all $n+2 \leq \ell\leq 2n+1$.
	
	When $\ell=n+1$, by the formula in \cite[Theorem 5.1 (1)]{GRS11} we have
	$$\begin{aligned}
	&\ J_{\psi_{n+1,\beta}}(\Ind_{Q_{2n}(F)}^{\SO_{4n+3}(F)}\tau'\otimes\sigma)\\
	\equiv & \ \delta_{\beta}\cdot\left(\ind_{Q'_{n,+}(F)}^{\SO_{2n,\beta}(F)}|\det(\cdot)|^{\frac{1-n}{2}}{\tau'}^{(n)}\oplus \ind_{Q'_{n,-}(F)}^{\SO_{2n,\beta}(F)}|\det(\cdot)|^{\frac{1-n}{2}}{\tau'}^{(n)}\right),
	\end{aligned}$$
	here $\delta_\beta=1$ if $\beta\in {(F^\times)}^2$, and is zero otherwise, and $Q'_{n,\pm}$ are defined as in \cite[Remark 5.1]{GRS11}.
	Note that $|\cdot|^{\frac{1-n}{2}}\tau'^{(n)}=\mu_1\times\cdots\times\mu_n$, we obtain all the statements for $n+1\leq \ell \leq 2n$ in part (1) and (2) of the proposition.
		
	Finally we take $\ell=n$. Also by \cite[Theorem 5.1 (1)]{GRS11}, we have 
	\[
	\begin{aligned}
	J_{\psi_n,\beta}(\Ind_{Q_{2n}(F)}^{\SO_{4n+3}(F)}\tau'\otimes\sigma)&\equiv \left(\ind_{Q'_{n}(F)}^{\SO_{2n+2,\beta}(F)}\lvert\cdot\rvert^{\frac{1-n}{2}}\tau'^{(n)}\otimes J_{\psi'_{0,\beta}}(\sigma)\right)\oplus\\
	&\quad \quad\delta_{\beta} \cdot \bigg(
	\ind_{Q'_{n,+}(F)}^{\SO_{2n+2,\beta}(F)}|\det(\cdot)|^{\frac{2-n}{2}}{\tau'}^{(n-1)}\otimes J_{\mathrm{Wh},+}(\sigma) \\
	& \quad \quad \oplus \ind_{Q'_{n,-}(F)}^{\SO_{2n+2,\beta}(F)}|\det(\cdot)|^{\frac{2-n}{2}}{\tau'}^{(n-1)}\otimes J_{\mathrm{Wh},-}(\sigma)\bigg)\, ,
	\end{aligned}\]
	where $J_{\mathrm{Wh},\pm}(\sigma)$ are Jacquet modules with respect to certain Whittaker characters (see \cite[Proposition 5.5]{GRS11}).
	Here we have used the fact that $\tau'_{(n)}=0$ (see, for example, \cite[page 99]{GRS11}), and $Q'_{n}$ is defined in \cite[(5.16)]{GRS11}. Note also that here $Q_n'$ contains $B_{\SO_{2n+2,\beta}}$.
	Moreover, we have $J_{\mathrm{Wh},\pm}(\sigma)\simeq \BC$ since $\sigma$ is generic. Since $\delta_\beta=0$ if $\beta\notin {(F^\times)}^2$, we can conclude (1) from the above formula. And if $\beta\in {(F^\times)}^2$, we take $(n-1)$-th derivative of $\tau'$ (as in \cite[\S 4]{BZ2}), and then the last expression stated in the proposition follows as desired. 
	The $\lambda$'s in Part (2) of the proposition come from the Jacquet module $J_{\psi'_{0,\beta}}(\sigma)$, which is exactly the restriction of $\sigma$ to the $F$-split torus $\SO_2(F)$. 
\end{proof}

\medskip

Now suppose that $\wt m=2n$, so that $V^{(2n)}=V_0$ is an anisotropic quadratic space of dimension $3$. Let
$\sigma$ be an irreducible admissible (finite dimensional) representation of $\SO_3^{V_0}(F)$. Given
$\tau$ be as in (\ref{unramified with trivial central character}), we also want to study the possible unramified constituents of the twisted Jacquet module
$J_{\psi_{\ell,\beta}}(\Ind_{P(F)}^{\SO^{V_0}_{4n+3}(F)} \tau|\cdot|^{1/2}\otimes\sigma)\,.$
Since we are interested in the unramified constituents, we only need to consider $\beta\in F^\times$ such that the target group $G_{2n-\ell+1,\beta}$
is $F$-quasi-split.

\begin{prop}\label{prop for unramified constituents I.2}
	Let $\tau$ and $\sigma$ be as above, and take $\beta\in F^\times$ such that $G_{2n-\ell+1,\beta}$ is quasi-split over $F$.
	Then any unramified constituent of $J_{\psi_{\ell,\beta}}(\Ind_{P(F)}^{\SO^{V_0}_{4n+3}(F)} \tau|\cdot|^{1/2}\otimes\sigma)$ is zero for $n+1 \leq \ell\leq 2n-1$. When $\ell=n$, any unramified constituent of $J_{\psi_{\ell,\beta}}(\Ind_{P(F)}^{\SO^{V_0}_{4n+3}(F)} \tau|\cdot|^{1/2}\otimes\sigma)$ is a subquotient of
	$\Ind_{B_{\SO_{2n+2,\beta}}(F)}^{\SO_{2n+2,\beta}(F)} \mu_1\otimes\cdots\otimes\mu_n\otimes 1\,.$
\end{prop}

\begin{proof}
	We will also use \cite[Theorem 5.1]{GRS11}. By conjugation of some Weyl element, it suffices to consider the unramified constituent of $J_{\psi_{\ell,\beta}}(\Ind_{P(F)}^{\SO^{V_0}_{4n+3}(F)} \tau'\otimes\sigma)$, where
	$$\tau'=\Ind_{P_{2,\cdots,2}(F)}^{\GL_{2n}(F)}\mu_1(\det\!_{\GL_2})\otimes\cdots\otimes\mu_{n}(\det\!_{\GL_2})\,.$$
	Applying \cite[Theorem 5.1 (2)]{GRS11} with $j=2n=\wt m$, we can see that $J_{\psi_{\ell,\beta}}(\Ind_{P(F)}^{\SO^{V_0}_{4n+3}(F)} \tau'\otimes\sigma)=0$ for $n+1\leq \ell\leq 2n-1$.  When $\ell=n$, we have
	$$J_{\psi_{n,\beta}}(\Ind_{P(F)}^{\SO^{V_0}_{4n+3}(F)} \tau'\otimes\sigma)\equiv \delta^{1}_{\SO^{V_0}_{4n+3},\beta}\cdot  \ind_{Q_\beta'(F)}^{\SO^{V_0}_{2n+2,\beta}(F)} |\det(\cdot)|^{\frac{1-n}{2}}{\tau'}^{(n)}\otimes J_{\psi'_{0,v_\beta}}(\sigma)\,.$$
	Here $\delta^{1}_{\SO^{V_0}_{4n+3},\beta}=1$ (see \cite[Theorem 5.1]{GRS11}) only if $\beta\in F^\times$ is taken to be such that $G_{2n-\ell+1,\beta}$ is quasi-split over $F$, and $Q_\beta'$ is the same as in \cite[(5.30)]{GRS11}.
	Moreover, the Jacquet module $J_{\psi_0',v_\beta}(\sigma)$ is just the restriction to an $F$-quasi-split group $\SO^{V_0}_{2,\beta}(F)$.
	Then we obtain the last statement the proposition by taking $n$-th derivative of $\tau'$.
\end{proof}

\textbf{Case II: non-trivial central character.}
Let $\lambda_0$ be the unique non-trivial unramified quadratic character of $F^\times$, and let
\begin{equation}\label{unramified with non-trivial central character}
\tau=\mu_1\times\cdots\times \mu_{n-1} \times 1\times\lambda_0\times \mu_{n-1}^{-1}\times\cdots \times \mu_1^{-1}\,,
\end{equation}
where $\mu_i$'s are unramified characters of $F^\times$ (recall that $\tau$ is of orthogonal type).
We also first consider the case $\wt m =2n+1$. Let $\sigma=\Ind_{B_{\SO_3}(F)}^{\SO_3(F)} \xi$, where $\xi$ is also an unramified character of $F^\times$.
In this case, we have
\begin{prop}\label{prop for unramified constituents II.1}
The twisted Jacquet module $J_{\psi_{\ell,\beta}}(\pi_{\tau\otimes\sigma})=0$ for all $n+2 \leq \ell\leq 2n$.
Moreover, define that $\delta_\beta=1$ if $\beta\in {(F^\times)}^2$ and $\delta_\beta=0$ otherwise, 
and define that $\delta^0_\beta=1$ if $\beta\notin {(F^\times)}^2$ and $\delta^0_\beta=0$ otherwise. Then
	\begin{enumerate}
		\item When $\ell=n+1$, any unramified constituent of $J_{\psi_{n+1,\beta}}(\pi_{\tau\otimes\sigma})$ is a subquotient of
		$$\begin{aligned}
		&\quad \ \ \delta^0_\beta\cdot  \Ind_{B_{\SO_{2n,\beta}}(F)}^{\SO_{2n,\beta}(F)} \mu_1\otimes\cdots\otimes\mu_{n-1}\otimes  1
\,.
		\end{aligned}$$
		
		\item When $\ell=n$, any unramified constituent of $J_{\psi_{n,\beta}}(\pi_{\tau\otimes\sigma})$ is a subquotient of
	$$\begin{aligned}
	    &\quad \bigoplus_{t=1}^{n-1} \Ind_{P^{1,\cdots,1,2,1,\cdots,1}_{\SO_{2n+2,\beta}}(F)}^{\SO_{2n+2,\beta}(F)} \mu_1\otimes\cdots\otimes \mu_{t-1}\otimes\mu_{t}(\det\!_{\GL_2})\otimes\mu_{t+1}\otimes\cdots\otimes\mu_{n-1}\otimes 1\\
        &\quad \bigoplus_{t=1}^{n-1} \delta_\beta\cdot\Ind_{P^{1,\cdots,1,2,1,\cdots,1}_{\SO_{2n+2,\beta}}(F)}^{\SO_{2n+2,\beta}(F)} \mu_1\otimes\cdots\otimes \mu_{t-1}\otimes\mu_{t}(\det\!_{\GL_2})\otimes\mu_{t+1}\otimes\cdots\otimes\mu_{n-1}\otimes \lambda_0\\
		&\quad \oplus \Ind_{B_{\SO_{2n+2,\beta}}(F)}^{\SO_{2n+2,\beta}(F)} \mu_1\otimes \cdots \otimes\mu_{n-1}\otimes \xi\otimes  1\\
		&\quad \oplus \Ind_{B_{\SO_{2n+2,\beta}}(F)}^{\SO_{2n+2,\beta}(F)} \mu_1\otimes \cdots \otimes\mu_{n-1}\otimes \lambda_0|\cdot |^{\frac{1}{2}}\otimes  1\\
		&\quad \oplus \delta_\beta\cdot\bigg(\Ind_{B_{\SO_{2n+2,\beta}}(F)}^{\SO_{2n+2,\beta}(F)} \mu_1\otimes \cdots \otimes\mu_{n-1}\otimes \xi\otimes  \lambda_0 \\
		&\qquad \qquad \qquad \oplus \Ind_{B_{\SO_{2n+2,\beta}}(F)}^{\SO_{2n+2,\beta}(F)} \mu_1\otimes \cdots \otimes\mu_{n-1}\otimes |\cdot|^{\frac{1}{2}}\otimes  \lambda_0 \bigg)\\
		&\quad \oplus \delta_\beta^0\cdot\Ind_{B_{\SO_{2n+2,\beta}}(F)}^{\SO_{2n+2,\beta}(F)} \mu_1\otimes \cdots \otimes\mu_{n-1}\otimes |\cdot|^{\frac{1}{2}} \otimes 1 \,.
		\end{aligned}$$		
	\end{enumerate}
\end{prop}

\begin{proof}
    By conjugation of some Weyl element, it suffices to consider the unramified constituent of the induced representation $\Ind_{Q_{2n-2}(F)}^{\SO_{4n+3}(F)}\tau_1'\otimes\sigma_1$, here
	$$\tau_1'=\Ind_{P_{2,\cdots,2,1,1}(F)}^{\GL_{2n}(F)}\mu_1(\det\!_{\GL_2})\otimes\cdots\otimes\mu_{n-1}(\det\!_{\GL_2})\,,$$
	and $\sigma_1=\Ind_{Q_2^{\SO_7}(F)}^{\SO_7(F)} \wt \tau_1 \otimes \mathbf{1}_{\SO_3}$, with $\wt \tau_1=\lambda_0|\cdot|^{1/2}\times \xi$ being a representation of $\GL_2(F)$.
	Note that $\sigma_1$ is {\it non-generic} and $\tau'^{(\ell)}=0$ for $\ell\geq n$. Applying \cite[Theorem 5.1 (1)]{GRS11} with $j=2n-2<\wt m=2n+1$, we see that if
	$n+2\leq \ell\leq 2n$, this twisted Jacquet module is always zero.
	
	Now we consider the case $\ell=n+1$. 
	Applying \cite[Theorem 5.1 (1)]{GRS11} with $j=2n-2<\wt m=2n+1$, we have 
		\[
	\begin{aligned}
	J_{\psi_{n+1,\beta}}(\Ind_{Q_{2n-2}(F)}^{\SO_{4n+3}(F)}\tau_1'\otimes\sigma_1)&\equiv \ind_{Q'_{n-3}(F)}^{\SO_{2n,\beta}(F)}|\det(\cdot)|^{-\frac{n-2}{2}}\tau'^{(n-1)}\otimes J_{\psi'_{2,\beta}}(\sigma_1)\,.
	\end{aligned}\]
	To calculate the Jacquet module $J_{\psi'_{2,\beta}}(\sigma_1)$, 
	applying \cite[Theorem 5.1 (1)]{GRS11} again with $j=2$ and $\ell=2$, we have 
	$$J_{\psi'_{2,\beta}}(\sigma_1)\equiv \ind_{Q_{0}'(F)}^{\SO_{2,\beta}(F)} |\cdot|^{-\frac{1}{2}}\wt\tau_1^{(2)}\otimes J_{\psi'_{0,\beta}}(\mathbf{1}_{\SO_3})\,.$$
	Hence by taking derivatives $\tau_1'^{(n-1)}$ and $\wt \tau_1^{(2)}$, we can see that any unramified constituent of $J_{\psi_{n+1,\beta}}(\Ind_{Q_{2n-2}(F)}^{\SO_{4n+3}(F)}\tau_1'\otimes\sigma_1)$ is a subquotient of
	\begin{equation}\label{equation 1 for Prop 3.3}
	\Ind_{B_{\SO_{2n,\beta}}(F)}^{\SO_{2n,\beta}(F)} \mu_1\otimes\cdots\otimes\mu_{n-1}\otimes 1\, .		
	\end{equation}
	On the other hand, we may also write $\sigma_1=\Ind_{Q_2^{\SO_7}(F)}^{\SO_7(F)} \wt \tau_2 \otimes \sigma_{\lambda_0}$, with $\wt \tau_2=|\cdot|^{1/2}\times \xi$ and $\sigma_{\lambda_0}$ is the unique irreducible quotient in $\Ind_{B_{\SO_3}(F)}^{\SO_3(F)}\lambda_0|\cdot|^{1/2}$. 
	Recall that $\lambda_0$ is a quadratic character, then $\sigma_{\lambda_0}$ can be viewed as a character of $\PGL_2(F)$ by composing it with the determinant, modulo squares. If $\beta\notin (F^\times)^2$, similarly we can see that any unramified constituent of $J_{\psi_{n+1,\beta}}(\Ind_{Q_{2n-2}(F)}^{\SO_{4n+3}(F)}\tau_1'\otimes\sigma_1)$ also lies in (\ref{equation 1 for Prop 3.3}). 
	If $\beta\in (F^\times)^2$, applying \cite[Theorem 5.1 (2)]{GRS11} with $j=1$ and $\ell=0$, we can see that  
	$$J_{\psi'_{0,\beta}}(\sigma_{\lambda_0})\equiv \lambda_0, $$
    and hence any unramified constituent of $J_{\psi_{n+1,\beta}}(\Ind_{Q_{2n-2}(F)}^{\SO_{4n+3}(F)}\tau_1'\otimes\sigma_1)$ is a subquotient of
	\begin{equation}\label{equation 2 for Prop 3.3}
	\Ind_{B_{\SO_{2n,\beta}}(F)}^{\SO_{2n,\beta}(F)} \mu_1\otimes\cdots\otimes\mu_{n-1}\otimes \lambda_0\, .
	\end{equation}
	Hence when $\beta\in (F^\times)^2$, any unramified constituent of $J_{\psi_{n+2,\beta}}(\pi_{\tau\otimes\sigma})$ is a subquotient of both (\ref{equation 1 for Prop 3.3}) and (\ref{equation 2 for Prop 3.3}). 
	If it is non-zero, then the two sets $\{\mu_1^{\pm}, \mu_2^{\pm}, \cdots, \mu_{n-1}^{\pm},1\}$ and $\{\mu_1^{\pm}, \mu_2^{\pm}, \cdots, \mu_{n-1}^{\pm},\lambda_0\}$ will be equal, which is impossible. 
	This shows that any irreducible constituent must be $0$, and this finishes the proof of Part (1) of the proposition.
	
	Finally we take $\ell=n$.
	Applying \cite[Theorem 5.1 (1)]{GRS11} with $j=2n-2<\wt m=2n+1$, we have 
	\begin{equation}\label{equation 3 for Prop 3.3}
	\begin{aligned}
	J_{\psi_{n,\beta}}(\Ind_{Q_{2n-1}(F)}^{\SO_{4n+3}(F)} \tau_1'\otimes \sigma_1)
	\equiv \, & \ind_{Q_n'(F)}^{\SO_{2n+2,\beta}(F)} |\det(\cdot)|^{\frac{3-n}{2}} \tau_1'^{(n-2)}\otimes J_{\psi'_{2,\beta}}(\sigma_1)\\
	&\oplus  \ind_{Q_{n-1}'(F)}^{\SO_{2n+2,\beta}(F)} |\det(\cdot)|^{\frac{2-n}{2}}\tau_1'^{(n-1)}\otimes J_{\psi'_{1,\beta}}(\sigma_1)\,.
	\end{aligned}
	\end{equation}
	Now, arguing as in the case of $\ell=n+1$, we can see that if an unramified constituent of $J_{\psi_{n,\beta}}(\Ind_{Q_{2n-1}(F)}^{\SO_{4n+3}(F)} \tau_1'\otimes \sigma_1)$ comes from the first summand of (\ref{equation 3 for Prop 3.3}), then it is a subquotient of 
	$$\begin{aligned}
	    &\quad \bigoplus_{t=1}^{n-1} \Ind_{P^{1,\cdots,1,2,1,\cdots,1}_{\SO_{2n+2,\beta}}(F)}^{\SO_{2n+2,\beta}(F)} \mu_1\otimes\cdots\otimes \mu_{t-1}\otimes\mu_{t}(\det\!_{\GL_2})\otimes\mu_{t+1}\otimes\cdots\otimes\mu_{n-1}\otimes 1\\
        &\quad \bigoplus_{t=1}^{n-1} \delta_\beta\cdot\Ind_{P^{1,\cdots,1,2,1,\cdots,1}_{\SO_{2n+2,\beta}}(F)}^{\SO_{2n+2,\beta}(F)} \mu_1\otimes\cdots\otimes \mu_{t-1}\otimes\mu_{t}(\det\!_{\GL_2})\otimes\mu_{t+1}\otimes\cdots\otimes\mu_{n-1}\otimes \lambda_0\,.
		\end{aligned}$$
	To compute the Jacquet module $J_{\psi'_{1,\beta}}(\sigma_1)$ in the second summand of (\ref{equation 3 for Prop 3.3}), we apply \cite[Theorem 5.1 (1)]{GRS11} with $j=2$ and $\ell=1$, and get 		
	$$J_{\psi'_{1,\beta}}(\sigma_1)\equiv \ind_{Q_{1}'(F)}^{\SO_{4,\beta}(F)} \wt\tau_1^{(1)}\otimes J_{\psi'_{0,\beta}}(\mathbf{1}_{\SO_3}).$$
	Moreover, one may also write $\sigma_1=\Ind_{Q_2^{\SO_7}(F)}^{\SO_7(F)} \wt \tau_2 \otimes \sigma_{\lambda_0}$ and get 
	$$J_{\psi'_{1,\beta}}(\sigma_1) \equiv \ind_{Q_{1}'(F)}^{\SO_{4,\beta}(F)} \wt\tau_2^{(1)}\otimes J_{\psi'_{0,\beta}}(\sigma_{\lambda_0}).$$
	Recall that we have $J_{\psi'_{0,\beta}}(\sigma_{\lambda_0})\equiv \lambda_0$ if $\beta\in (F^\times)^2$. 
	Then we obtain that, if an unramified constituent of $J_{\psi_{n,\beta}}(\Ind_{Q_{2n-1}(F)}^{\SO_{4n+3}(F)} \tau_1'\otimes \sigma_1)$ comes from the second summand of (\ref{equation 3 for Prop 3.3}), then it is a subquotient of 
			$$
		\begin{aligned}
        &\quad \Ind_{B_{\SO_{2n+2,\beta}}(F)}^{\SO_{2n+2,\beta}(F)} \mu_1\otimes \cdots \otimes\mu_{n-1}\otimes \xi\otimes  1\\ 
        &\quad \oplus \Ind_{B_{\SO_{2n+2,\beta}}(F)}^{\SO_{2n+2,\beta}(F)} \mu_1\otimes \cdots \otimes\mu_{n-1}\otimes \lambda_0|\cdot |^{\frac{1}{2}}\otimes  1\\
		&\quad \oplus \delta_\beta\cdot\bigg(\Ind_{B_{\SO_{2n+2,\beta}}(F)}^{\SO_{2n+2,\beta}(F)} \mu_1\otimes \cdots \otimes\mu_{n-1}\otimes \xi\otimes  \lambda_0 \\
		&\qquad \qquad \qquad \oplus \Ind_{B_{\SO_{2n+2,\beta}}(F)}^{\SO_{2n+2,\beta}(F)} \mu_1\otimes \cdots \otimes\mu_{n-1}\otimes |\cdot|^{\frac{1}{2}}\otimes  \lambda_0 \bigg)\\
		&\quad \oplus \delta_\beta^0\cdot\Ind_{B_{\SO_{2n+2,\beta}}(F)}^{\SO_{2n+2,\beta}(F)} \mu_1\otimes \cdots \otimes\mu_{n-1}\otimes |\cdot|^{\frac{1}{2}} \otimes 1 \,,
		\end{aligned}
		$$	
	as desired. 
\end{proof}

\medskip

We still need to consider the case that $\wt m=2n$, where $\sigma$ is an irreducible admissible representation of a non-split group $\SO_3^{V_0}(F)$ over $F$.
\begin{prop}\label{prop for unramified constituents II.2}
	Let $\tau$ be as in (\ref{unramified with non-trivial central character}), $\sigma$ be as above, and take $\beta\in F^\times$ such that $G_{2n-\ell+1,\beta}$ is quasi-split over $F$.
	Then any unramified constituent of the twisted Jacquet module 
	$J_{\psi_{\ell,\beta}}(\Ind_{P(F)}^{\SO^{V_0}_{4n+3}(F)} \tau|\cdot|^{1/2}\otimes\sigma)$ is zero for $n+2 \leq \ell\leq 2n-1$. Moreover, 
	the following hold. 
	\begin{enumerate}
		\item  When $\ell=n+1$, any unramified constituent of $J_{\psi_{n+1,\beta}}(\Ind_{P(F)}^{\SO^{V_0}_{4n+3}(F)} \tau|\cdot|^{1/2}\otimes\sigma)$ is a subquotient of
		$\Ind_{B_{\SO_{2n,\beta}}(F)}^{\SO_{2n,\beta}(F)} \mu_1\otimes\cdots\otimes\mu_{n-1}\otimes 1\,.$
		\item When $\ell=n$,
		any unramified constituent of $J_{\psi_{n,\beta}}(\Ind_{P(F)}^{\SO^{V_0}_{4n+3}(F)} \tau|\cdot|^{1/2}\otimes\sigma)$ is a subquotient of
		$$\begin{aligned}
		& \quad \ \Ind_{B_{\SO_{2n+2,\beta}}(F)}^{\SO_{2n+2,\beta}(F)} \mu_1\otimes\cdots\otimes\mu_{n-1}\otimes|\cdot|^{\frac{1}{2}}\otimes1\\
		& \oplus \Ind_{B_{\SO_{2n+2,\beta}}(F)}^{\SO_{2n+2,\beta}(F)} \mu_1\otimes\cdots\otimes
		\mu_{n-1}\otimes\lambda_0|\cdot|^{\frac{1}{2}}\otimes1 \\
		&\bigoplus_{t=1}^{n-1} \Ind_{P^{1,\cdots,1,2,1,\cdots,1}_{\SO_{2n+2,\beta}}(F)}^{\SO_{2n+2,\beta}(F)} \mu_1\otimes\cdots\otimes \mu_{t-1}\otimes\mu_{t}(\det\!_{\GL_2})\otimes\mu_{t+1}\otimes\cdots\otimes\mu_{n-1}\otimes 1\,.
		\end{aligned}$$
	\end{enumerate}
\end{prop}
\begin{proof}
    As before, it suffices to consider the unramified constituent of the induced representation $\Ind_{Q_{2n-2}(F)}^{\SO^{V_0}_{4n+3}(F)}\tau_1'\otimes\sigma_1$, here
	$$\tau_1'=\Ind_{P_{2,\cdots,2}(F)}^{\GL_{2n}(F)}\mu_1(\det\!_{\GL_2})\otimes\cdots\otimes\mu_{n-1}(\det\!_{\GL_2})\,.$$
	and $\sigma_1=\Ind_{Q_2^{\SO^{V_0}_7}(F)}^{\SO^{V_0}_7(F)}\tau_2\otimes\sigma$ with $\tau_2=\lambda_0|\cdot|^{1/2}\times |\cdot|^{1/2}$. Note that $\sigma_1$ is non-generic. 
    Applying \cite[Theorem 5.1 (2)]{GRS11} with $j=2n-2$, we can see that this twisted Jacquet module is zero if $n+2\leq \ell\leq 2n-1$, since ${\tau_1'}^{(\ell)}=0$ for $\ell\geq n$. 
    
    When $\ell=n+1$, also by \cite[Theorem 5.1 (1)]{GRS11}, we have 
    $$J_{\psi_{n+1,\beta}}(\Ind_{P(F)}^{\SO^{V_0}_{4n+3}(F)} \tau_1'\otimes\sigma_1)\equiv \ind_{Q_\beta'(F)}^{\SO_{2n,\beta}(F)} |\det(\cdot)|^{-\frac{n}{2}}{\tau_1'}^{(n-1)}\otimes J_{\psi'_{2,v_\beta}}(\sigma_1)\,.$$
    Note that we have assumed that $G_{2n-\ell+1,\beta}$ is $F$-quasi-split.
    Consider the Jacquet module $J_{\psi'_{2,v_\beta}}(\sigma_1)$. Applying \cite[Theorem 5.1 (2)]{GRS11} with $\ell=2$ and $j=2$, we have 
    $$J_{\psi'_{2,v_\beta}}(\sigma_1)\equiv \ind_{Q^{\SO_7^{V_0}\,\prime}_\beta(F)}^{\SO_{2,\beta}(F)}|\cdot|^{-1} \tau_2^{(2)}\otimes J_{\psi'_{0,v_\beta}}(\sigma).$$
    The last Jacquet module is just a restriction to the anisotropic $\SO_{2,\beta}(F)$. This proves Part (1) of the proposition. 
    
    Finally we take $\ell=n$. Applying \cite[Theorem 5.1 (1)]{GRS11} with $j=2n-2$, we have 
    $$\begin{aligned}
	J_{\psi_{n,\beta}(F)}(\Ind_{P(F)}^{\SO^{V_0}_{4n+3}(F)} \tau'\otimes\sigma)
	&\equiv \ind_{Q'_n(F)}^{\SO_{2n+2,\beta}(F)} |\det(\cdot)|^{\frac{2-n}{2}} \tau_1'^{(n-1)}\otimes J_{\psi'_{1,\beta}}(\sigma_1)\\
	&\oplus \ind_{Q_\beta'(F)}^{\SO_{2n+2,\beta}(F)} |\det(\cdot)|^{\frac{3-n}{2}}{\tau_1'}^{(n-2)}\otimes J_{\psi'_{2,v_\beta}}(\sigma_1).
	\end{aligned}$$
    We still need to calculate the Jacquet module $J_{\psi'_{1,\beta}}(\sigma_1)$. Applying \cite[Theorem 5.1 (2)]{GRS11} again with $j=2$ and $\ell=1$, we have 
    $$J_{\psi'_{1,\beta}}(\sigma_1)\equiv \left(\ind_{Q^{\SO_7^{V_0}\, \prime}_1(F)}^{\SO_{4,\beta}(F)} |\cdot|^{-\frac{1}{2}}(\tau_2)_{(1)}\otimes \sigma \right)  \oplus\left(\ind_{Q^{\SO_7^{V_0}\,\prime}_\beta(F)}^{\SO_{4,\beta}(F)}\tau_2^{(1)}\otimes J_{\psi'_{0,v_\beta}}(\sigma)\right).$$
    Then, granting the above calculation of $J_{\psi'_{2,v_\beta}}(\sigma_1)$, we obtain Part (2) of the proposition by considering the unramified constituents. 
\end{proof}

\medskip

We consider moreover the special case that $\ell=\wt m$, and $w_0\in W$. The result is the following: 

\begin{prop}\label{w_0 in W}
	Let $\tau$ and $\sigma$ be the same as in any case of Propositions \ref{prop for unramified constituents I.1}--\ref{prop for unramified constituents II.2}, and let $w_0\in W$. Then we have
	$J_{\psi_{\wt m, w_0}}(\pi_{\tau\otimes\sigma})=0\,.$
\end{prop}
\begin{proof}
	As before, conjugating by some Weyl element, it suffices to consider the unramified constituent of $\displaystyle{J_{\psi_{\wt m, w_0}}(\Ind_{P(F)}^{H(F)}(\tau'\otimes\sigma))}$,
	where
	$$\tau'=\Ind_{P_{2,\cdots,2}}^{\GL_{2n}(F)}\mu_1(\det\!_{\GL_2})\otimes\cdots\otimes\mu_{n}(\det\!_{\GL_2})\,,$$
	or
	$$\tau'=\Ind_{P_{2,\cdots,2,1,1}(F)}^{\GL_{2n}(F)}\mu_1(\det\!_{\GL_2})\otimes\cdots\otimes\mu_{n-1}(\det\!_{\GL_2})\otimes\lambda_0|\cdot|^{1/2}\otimes|\cdot|^{1/2}\,.$$
	By \cite[Theorem 5.1 (3)]{GRS11}, we have
	$$J_{\psi_{\wt m, w_0}}(\Ind_{P(F)}^{H(F)}(\tau'\otimes\sigma))=d_{\tau'}\cdot J_{\psi_{\wt m-2n}}(\sigma)\,,$$
	where $d_{\tau'}$ is the dimension of the space of $\psi$-Whittaker functionals on $\tau'$.
	By the construction of $\tau'$, we can see that $d_{\tau'}=0$, and the proposition follows.
\end{proof}

\bigskip

At the end of this section, we consider an additional case which will be used in the proof of Proposition \ref{prop for cuspidality} later. In this case, we consider the induced representation 
$$\Ind_{R_{2n-1}}^{\SO_{4n+1}(F)} \tau_1|\cdot|^{1/2}\otimes\sigma$$
of the $F$-split group $\SO_{4n+1}(F)$.  
Here $R_{2n-1}\subset \SO_{4n+1}$ is the maximal parabolic subgroup whose Levi subgroup is isomorphic to $\GL_{2n-1}\times \SO_3$, $\sigma=\Ind_{B_{\SO_3}(F)}^{\SO_3(F)} \xi$ with $\xi$ being an unramified character of $F^\times$, and
\begin{equation}\label{unramified for the case 2n-1}
\tau_1=\mu_1\times\cdots\times \mu_{n-1} \times\lambda \times \mu_{n-1}^{-1}\times\cdots \times \mu_1^{-1}\,,
\end{equation} 
with $\mu_i$'s and $\lambda$ being unramified characters of $F^\times$ and in addition $\lambda$ being quadratic. 
For $1 \leq \ell\leq 2n$, let $N_\ell^\sharp\subset \SO_{4n+1}$ and $\psi_{\ell,\beta}^\sharp: N_\ell^\sharp(F)\lra \BC^\times$ be similar to  $N_\ell\subset H$ and $\psi_{\ell,\beta}: N_\ell(F)\lra \BC^\times$ in previous parts of this section. 
Now let $\pi_{\tau_1\otimes\sigma}$ be the unramified constituent of $\Ind_{R_{2n-1}}^{\SO_{4n+1}(F)} \tau_1|\cdot|^{1/2}\otimes\sigma$, we are going to consider the twisted Jacquet module 
$J_{\psi^\sharp_{\ell,\beta}}(\pi_{\tau_1\otimes\sigma})\, .$

\begin{prop}\label{prop for unramified constituents: additional}
 The following hold:
 	\begin{enumerate}
		\item Suppose that $\beta\in F^\times-{(F^\times)}^2$. Then any twisted Jacquet module $J_{\psi^\sharp_{\ell,\beta}}(\pi_{\tau_1\otimes\sigma})=0$ for all $n \leq \ell\leq 2n$. 		
	    \item Suppose that $\beta \in {(F^\times)}^2$. Then any twisted Jacquet module $J_{\psi^\sharp_{\ell,\beta}}(\pi_{\tau_1\otimes\sigma})=0$ for all $n+1 \leq \ell\leq 2n$. When $\ell=n$, any unramified constituent of $J_{\psi^\sharp_{n+1,\beta}}(\pi_{\tau_1\otimes\sigma})$ is a subquotient of
		$\Ind_{B_{\SO_{2n}}(F)}^{\SO_{2n}(F)} \mu_1\otimes\cdots\otimes\mu_{n-1}\otimes \lambda\,.$
    \end{enumerate}
\end{prop}

\begin{proof}
    The calculation is similar to those in other cases. 
    By conjugation of some Weyl element, it suffices to consider the unramified constituent of the induced representation $\Ind_{R_{2n-2}(F)}^{\SO_{4n+1}(F)}\tau_1'\otimes\sigma'$ instead of $\pi_{\tau\otimes\sigma}$, here
	$$\tau_1'=\Ind_{P_{2,\cdots,2}(F)}^{\GL_{2n-2}(F)}\mu_1(\det\!_{\GL_2})\otimes\cdots\otimes\mu_{n-1}(\det\!_{\GL_2})\,,$$
	and $\sigma'=\Ind_{P_{\SO_5}^1(F)}^{\SO_5(F)}\xi\otimes \sigma_\lambda$. As before, we have deonted by $\sigma_{\lambda}$ the unique irreducible quotient of  $\Ind_{B_{\SO_3}(F)}^{\SO_3(F)}\lambda|\cdot|^{1/2}$, which can be viewed as a character of $\PGL_2(F)$ by composing it with the determinant, modulo squares.
	Note that $\sigma'$ is non-generic and the derivative $\tau_1'^{(\ell)}$ of $\tau_1'$ vanishes for $\ell\geq n$. Now applying \cite[Theorem 5.1 (i)]{GRS11} with $j=2n-2$, we can see that
	the corresponding twisted Jacquet module vanishes for all $n+1 \leq \ell\leq 2n$.
	
	When $\ell=n$, by the formula in \cite[Theorem 5.1 (i)]{GRS11} we have
	$$\begin{aligned}
	&\quad\ \ J_{\psi^\sharp_{n,\beta}}(\Ind_{R_{2n-2}(F)}^{\SO_{4n+1}(F)}\tau_1'\otimes\sigma')\equiv \ind_{R'_{n-1}(F)}^{\SO_{2n,\beta}(F)}\lvert\cdot\rvert^{\frac{2-n}{2}}\tau_1'^{(n-1)}\otimes J_{{\psi^{\sharp\,\prime}_{1,\beta}}}(\sigma')\, .
	\end{aligned}$$
	Now we consider $J_{{\psi^{\sharp\,\prime}_{1,\beta}}}(\sigma')$. Applying \cite[Theorem 5.1 (i)]{GRS11} again with $\ell=1$ and $j=1$, we get 
	$$J_{{\psi^{\sharp\,\prime}_{1,\beta}}}(\sigma')\equiv J_{{\psi^{\sharp\,\prime}_{0,\beta}}}(\sigma_\lambda)\, ,$$ 
	which is zero unless $\beta\in (F^\times)^2$, in which case we have $J_{{\psi^{\sharp\,\prime}_{0,\beta}}}(\sigma_\lambda)\equiv \lambda$. Then the desired results follow.	
\end{proof}

\section{The cuspidality of the descent}\label{cuspidality}

In the section, we show the cuspidality of the descent representation. We come back to the global settings and start with the following lemma, which is a direct consequence of the local calculations in the previous section (Propositions \ref{prop for unramified constituents I.1} and \ref{w_0 in W}). Throughout this section, we let $\tau$ and $\sigma$ be as in \S \ref{subsection: residual rep}. 

\begin{lem}\label{vanishing depth}
The Bessel-Fourier coefficients $f^{\psi_{\ell,w_0}}$ vanish for any anisotropic vector $w_0\in V^{(\ell)}$ and $f\in \CE_{\tau\otimes\sigma}$ while $n+2 \leq \ell \leq \wt m$.
\end{lem}

\begin{proof}
Write the residual representation as $\CE_{\tau\otimes\sigma}=\otimes_v' \Pi_v$. 
Note that $H(F_v)$ is $F_v$-split at almost all places $v$.
We fix a finite place $v$ such that $H(F_v)=\SO_{4n+3}(F_v)$ is $F_v$-split, $\tau_v$, $\sigma_v$, $\Pi_v$, $\psi_v$ are all unramified, and $\omega_{\tau_v}$ is trivial.
By a suitable conjugation we may assume that $w_0=y_{\beta}$ for some $\beta \in F^\times$ (see \cite{JLXZ14} or \cite[Lemma 2.4]{JZ15}), and
consider the corresponding twisted Jacquet module $J_{\psi_{\ell,\beta}}(\Pi_v)$.
Write
\begin{equation}\label{equation 1 for depth n+1 case}
	\tau_v=\mu_{1,v}\times\cdots\times \mu_{n,v} \times \mu_{n,v}^{-1}\times\cdots \times \mu_{1,v}^{-1}\,,
\end{equation}
where $\mu_{i,v}$'s are unramified characters of $F_v^\times$.
By Propositions \ref{prop for unramified constituents I.1} Part (2) and Proposition \ref{w_0 in W}, the Jacquet module $J_{\psi_{\ell,\beta}}(\Pi_v)$ vanishes for any $n+2 \leq \ell \leq \wt m$, hence the Bessel-Fourier coefficient $f^{\psi_{\ell,w_0}}$ vanishes for any $n+2 \leq \ell \leq \wt m$. 
This completes the proof of the lemma. 
\end{proof}

To prove the cuspidality of the descent, we also need the vanishing property for the depth $\ell=n+1$. 

\begin{prop}\label{vanishing at depth n+1}
Suppose that $(\tau, \sigma)$ satisfies Assumption \ref{depth assumption}. Then the Bessel-Fourier coefficients $f^{\psi_{n+1,w_0}}$ vanish for any anisotropic vector $w_0\in V^{(\ell)}$ and $f\in \CE_{\tau\otimes\sigma}$.
\end{prop}
\begin{proof}
Assume that the Bessel-Fourier coefficient $f^{\psi_{n+1,\beta}}\neq 0$ for some choice of data. Let $\CD_{\psi_{n+1,\beta}}(\CE_{\tau\otimes\sigma})$ be the
$\SO_{2n,\beta}(\BA)$-span of the coefficients $f^{\psi_{n+1,\beta}}|_{\SO_{2n,\beta}(\BA)}$.
Note that $\SO_{2n,\beta}(\BA)$ is relevant to $\SO(V_0)(\BA)$.
Then by a formula for constant terms in \cite[Theorem 7.2]{GRS11} and a similar argument to that in the proof of \cite[Theorem 7.6]{GRS11} (see also the remarks in \cite[\S 7.4]{GRS11} and the proof of \cite[Proposition 6.3]{JZ15}), the automorphic module $\CD_{\psi_{n+1,\beta}}(\CE_{\tau\otimes\sigma})$ is cuspidal and has a direct sum decomposition of irreducible cuspidal automorphic representations. 

Let $\pi$ be any irreducible summand of $\CD_{\psi_{n+1,\beta}}(\CE_{\tau\otimes\sigma})$. We claim that
$\pi$ has the generic global Arthur parameter $\phi_\tau$, and lies in the global Vogan packet $\wt \Pi_{\phi_\tau}[\SO^*_{2n}]$.

We prove the claim case by case. 
Frist we assume that $\omega_\tau=1$. 
If $\beta\in (F^\times)^2$, then for almost all places $v$ of $F$, we have $H(F_v)$ is $F_v$-split, $\tau_v$, $\sigma_v$, $\Pi_v$, $\psi_v$ are all unramified, $\omega_{\tau_v}=1$, and $\beta\in (F_v^\times)^2$.
Fix any such place $v$, and write 
\begin{equation*}
	\tau_v=\mu_{1,v}\times\cdots\times \mu_{n,v} \times \mu_{n,v}^{-1}\times\cdots \times \mu_{1,v}^{-1}\,,
\end{equation*}
where $\mu_{i,v}$'s are unramified characters of $F_v^\times$.
By Proposition \ref{prop for unramified constituents I.1} Part (2), the unramified constituent of the Jacquet module $J_{\psi_{n+1,\beta}}(\Pi_v)$ is a subquotient of $$\Ind_{B_{\SO_{2n}}(F_v)}^{\SO_{2n}(F_v)} \mu_{1,v}\otimes\cdots\otimes\mu_{n,v}\,,$$
which shows that the Satake parameter of $\pi_v$ matches the Satake parameter of $\tau_v$, for almost all places. 
This shows that the global Arthur parameter $\phi_\pi=\phi_\tau$. 
If $\beta\notin(F^\times)^2$, then there exists a finite place $v$ such that $H(F_v)$ is $F_v$-split, $\tau_v$, $\sigma_v$, $\Pi_v$, $\psi_v$ are all unramified, $\omega_{\tau_v}=1$, and $\beta\notin (F_v^\times)^2$. 
Then by Proposition \ref{prop for unramified constituents I.1} Part (1), one sees that $\CD_{\psi_{n+1,\beta}}(\CE_{\tau\otimes\sigma})=0$, which is not the case we are considering. 

Next we assume that $\omega_\tau\neq 1$. If $\beta\in (F^\times)^2$, then there exists a finite place $v$ such that $H(F_v)$ is $F_v$-split, $\tau_v$, $\sigma_v$, $\Pi_v$, $\psi_v$ are all unramified, $\omega_{\tau_v}\neq 1$, and $\beta \in (F_v^\times)^2$. 
Then by Proposition \ref{prop for unramified constituents II.1} Part (1), one sees that $\CD_{\psi_{n+1,\beta}}(\CE_{\tau\otimes\sigma})=0$, which is also not the case we are considering. 
If $\beta\notin (F^\times)^2$, since we just consider the case that $\CD_{\psi_{n+1,\beta}}(\CE_{\tau\otimes\sigma})\neq 0$, then by Proposition \ref{prop for unramified constituents I.1} Part (2) and Proposition \ref{prop for unramified constituents II.1} Part (1), for the finite places $v$ such that $H(F_v)$ is $F_v$-split, $\tau_v$, $\sigma_v$, $\Pi_v$, $\psi_v$ are all unramified, 
we have either $\omega_{\tau_v}=1$ and $\beta\in (F_v^\times)^2$ or $\omega_{\tau_v}\neq 1$ and $\beta\notin (F_v^\times)^2$.
Now we consider the parameter of $\pi$. For the places such that $H(F_v)$ is $F_v$-split, $\tau_v$, $\sigma_v$, $\Pi_v$, $\psi_v$ are all unramified, $\omega_{\tau_v}=1$, and $\beta\in (F_v^\times)^2$, 
the argument for the first case ($\omega_\tau=1$ and $\beta\in (F^\times)^2$) already shows that the Satake parameters of $\tau_v$ and $\pi_v$ match each other. 
And for the places such that $H(F_v)$ is $F_v$-split, $\tau_v$, $\sigma_v$, $\Pi_v$, $\psi_v$ are all unramified, $\omega_{\tau_v}\neq 1$, and $\beta\notin (F_v^\times)^2$, we write 
\begin{equation}\label{equation 2 for depth n+1 case}
	\tau_v=\mu_{1,v}\times\cdots\times \mu_{n-1,v} \times 1\times \lambda_{0,v}\times \mu_{n-1,v}^{-1}\times\cdots \times \mu_{1,v}^{-1}\,,
\end{equation}
where $\lambda_{0,v}$ is an unramified quadratic character of $F_v^\times$. 
Then by Proposition \ref{prop for unramified constituents II.1} Part (1), the unramified constituent of the Jacquet module $J_{\psi_{n+1,\beta}}(\Pi_v)$ is a subquotient of $$\Ind_{B_{\SO_{2n}}(F_v)}^{\SO_{2n,\beta}(F_v)} \mu_{1,v}\otimes\cdots\otimes\mu_{n-1,v}\otimes 1\,,$$ which also shows that the Satake parameters of $\tau_v$ and $\pi_v$ match each other. 
Therefore, the Satake parameter of $\pi_v$ matches the Satake parameter of $\tau_v$ for almost all places, and hence $\phi_{\pi}=\phi_{\tau}$ by strong multiplicity one theorem for general linear groups.  

We continue the proof of the proposition. 
By construction, there exists an automorphic form $\varphi_\pi\in \pi$ such that the inner product
\begin{equation}\label{cuspequ1}
\pair{\varphi_\pi, \Res_{s=1/2}E^{\psi_{n+1,\beta}}(s, \cdot, \phi_{\tau\otimes\sigma})}\neq 0
\end{equation}
for some choice of data. By \cite[Corollary 4.4]{JZ15}, the pair $(\pi,\sigma)$ has a non-zero Bessel period, and hence is the GGP 
pair in the given global Vogan packet 
$\wt\Pi_{\phi_\pi \times \phi_{\tau_0}}[\SO^*_{2n}\times \SO^*_{3}]$ in Assumption \ref{depth assumption}. 
But this contradicts to Assumption \ref{depth assumption}. 
Hence the proposition is proved. 
\end{proof}
	
As a corollary, we have the cuspidality of the descent $\CD_{\psi_{n,\beta}}(\CE_{\tau\otimes\sigma})$.

\begin{prop}\label{prop for cuspidality}
Suppose that $(\tau, \sigma)$ satisfies Assumption \ref{depth assumption}. 
Then the twisted automorphic descent $\CD_{\psi_{n,\beta}}(\CE_{\tau\otimes\sigma})$ is a cuspidal automorphic $G_{n+1,\beta}(\BA)$-module.
\end{prop}

\begin{proof}
The proposition can be proved by considering all the constant terms of the Bessel-Fourier coefficient $f^{\psi_{n,\beta}}$ with
$f\in \CE_{\tau\otimes\sigma}$.
The proof is similar to that of \cite[Theorem 7.6]{GRS11}, see also
\cite[Proposition 2.6]{JLXZ14} and \cite[Proposition 6.3]{JZ15}. We briefly introduce it as follows. 

Let $\wt m_{n,\beta}$ be the Witt index of  $y_{\beta}^\perp\cap V^{(n)}$. 
Using \cite[Theorem 7.3]{GRS11}, and also the remarks in \cite[\S 7.4]{GRS11}, the constant terms $c_p(f^{\psi_{n,\beta}})$ $(1 \leq  p \leq  \wt m_{n,\beta})$, 
along maximal parabolic subgroup of $G_{n+1,\beta}$ with Levi subgroup isomorphic to $\GL_p \times G_{n+1-p, \beta}$, can be expressed as a summation whose summands are expressed as integrals of following terms
$$f^{\psi_{n+p,\beta}}\,, (f^{U_{p-i}})^{\psi_{n+i,\beta}}\,,
0 \leq i \leq p-1\,,$$
where $f^{U_{p-i}}$ is the constant term of $f$ along 
the maximal parabolic subgroup $Q_{p-i}\subset H$ with Levi subgroup isomorphic to {$\GL_{p-i} \times \SO(V^{(p-i)})$}. Here $(f^{U_{p-i}})^{\psi_{n+i,\beta}}$ denotes the 
$\psi_{n+i,\beta}$-Bessel-Fourier coefficient of 
{$f^{U_{p-i}}|_{\SO(V^{(p-i)})(\BA)}$}.
By Lemma \ref{vanishing depth} and Proposition \ref{vanishing at depth n+1}, $f^{\psi_{n+p,\beta}}$ vanishes since $p\geq 1$. By the cuspidal support of $\CE_{\tau\otimes\sigma}$, the constant term $f^{U_{p-i}}$ vanishes unless the parabolic $Q_{p-i}$ contains the cuspidal support $\tau_1 \otimes \cdots \otimes \tau_r \otimes \sigma$. It is easy to see that $f^{U_{p-i}}$ belongs to the representation 
$$(\tau_{i_1} \boxplus \cdots \boxplus \tau_{i_k}) \otimes \CE_{(\tau_{j_1} \boxplus \cdots \boxplus \tau_{j_{r-k}}) \otimes \sigma}\,,$$
where $\{i_1, \ldots, i_k\} \cup \{j_1, \ldots, j_{r-k}\} = \{1, 2, \ldots, r\}$, and $\CE_{(\tau_{j_1} \boxplus \cdots \boxplus \tau_{j_{r-k}}) \otimes \sigma}$ is the residual representation of $\SO(V^{(p-i)})(\BA)$
constructed in the same way as $\CE_{\tau\otimes\sigma}$. 
Note that $f^{U_{p-i}}|_{\SO(V^{(p-i)})(\BA)} \in \CE_{(\tau_{j_1} \boxplus \cdots \boxplus \tau_{j_{r-k}}) \otimes \sigma}$. 
Using a similar argument to that in Lemma \ref{vanishing depth} ({note that here $n+i$ is large enough by Assumption \ref{depth assumption}, Propositions \ref{prop for unramified constituents I.1} and \ref{prop for unramified constituents: additional}}), 
by Propositions \ref{prop for unramified constituents I.1} and \ref{prop for unramified constituents: additional} again, 
all $\psi_{n+i,\beta}$-Bessel-Fourier coefficients of $\CE_{(\tau_{j_1} \boxplus \cdots \boxplus \tau_{j_{r-k}}) \otimes \sigma}$ vanish, hence $(f^{U_{p-i}})^{\psi_{n+i,\beta}}$ also vanishes for any $0 \leq i \leq p-1$. 
Therefore, all the constant terms of $f^{\psi_{n,\beta}}$ along various unipotent subgroups of $G_{n+1,\beta}$ vanish, which means that the Bessel-Fourier coefficient $f^{\psi_{n,\beta}}$ is cuspidal, and the twisted automorphic descent $\CD_{\psi_{n,\beta}}(\CE_{\tau\otimes\sigma})$ is a cuspidal automorphic $G_{n+1,\beta}(\BA)$-module. 

This completes the proof of the proposition. 
\end{proof}


In following proposition, we show that in certain cases,
to obtain the cuspidality of the descent, we do not need to assume Assumption \ref{depth assumption}.

\begin{prop}\label{prop for special beta}
If we take $w_0=y_{\beta}$ with $\beta \notin (F^\times)^2$ when $\omega_\tau=1$, 
or take $w_0=y_{\beta}$ with $\beta \in (F^\times)^2$ when $\omega_\tau\neq 1$, then the Bessel-Fourier coefficient  $f^{\psi_{\ell,\beta}}$ vanishes for any $f\in \CE_{\tau\otimes\sigma}$ and any $n+1\leq \ell\leq \widetilde{m}$, 
and the descent $\CD_{\psi_{n,\beta}}(\CE_{\tau\otimes\sigma})$ is cuspidal.
\end{prop}

\begin{proof}
We have already known from Lemma \ref{vanishing depth} that $f^{\psi_{\ell,\beta}}$ vanishes for any $n+2 \leq \ell \leq \widetilde{m}$. 
Hence, for the first statement, we only need to consider the case of $\ell=n+1$. 
As we have seen in the proof of Proposition \ref{vanishing at depth n+1}, if $\beta\notin (F^\times)^2$ and $\omega_\tau=1$, there exists a finite place $v$ of $F$ such that $H(F_v)$ is $F_v$-split, $\tau_v$, $\sigma_v$, $\Pi_v$ and $\psi_v$ are are all unramified, $\beta\notin (F_v^\times)^2$ and $\omega_{\tau_v}=1$. 
And if $\beta\in (F^\times)^2$ and $\omega_\tau\neq 1$, there exists a finite place $v$ of $F$ such that $H(F_v)$ is $F_v$-split, $\tau_v$, $\sigma_v$, $\Pi_v$ and $\psi_v$ are all unramified, $\beta\in (F_v^\times)^2$ and $\omega_{\tau_v}\neq 1$. 
Then by Proposition \ref{prop for unramified constituents I.1} Part (1) and Proposition \ref{prop for unramified constituents II.1} Part (1), the global Fourier coefficient $f^{\psi_{n+1,\beta}}$ vanishes, as desired.

The cuspidality of the descent follows similarly along the same lines of the proof of Proposition \ref{prop for cuspidality}. 	
\end{proof}

\section{The non-vanishing of the descent construction}\label{non-vanishing}

In this section, we show that the descent $\CD_{\psi_{n,\beta}}(\CE_{\tau\otimes\sigma})$ is non-vanishing.

\subsection{Generalized and degenerate Whittaker-Fourier coefficients}
\label{FCdefinition}

First, we recall the generalized and degenerate Whittaker-Fourier coefficients attached to nilpotent orbits, following the formulation in \cite{GGS17}.
Let $\RG$ be a reductive group defined over a number field $F$.
Fix a nontrivial additive character $\psi: F\bs \BA \rightarrow \BC^{\times}$.
Let $\mathfrak{g}$ be the Lie algebra of $\RG(F)$ and $u$ be a nilpotent element in $\mathfrak{g}$.
The element $u$ defines a function on $\mathfrak{g}$:
\[
\psi_u: \mathfrak{g} \rightarrow \BC^{\times}
\]
by $\psi_u(x) = \psi(\kappa(u,x))$, where $\kappa$ is the Killing form on $\mathfrak{g}$.

Given any semi-simple element $s \in \mathfrak{g}$, under the adjoint action, $\mathfrak{g}$ is decomposed to a direct sum of eigenspaces $\mathfrak{g}^s_i$ corresponding to eigenvalues $i$.
The element
$s$ is called {\it rational semi-simple} if all its eigenvalues are in $\BQ$.
Given a nilpotent element $u$, a {\it Whittaker pair} is a pair $(s,u)$ with $s \in \mathfrak{g}$ being a rational semi-simple element, and $u \in \mathfrak{g}^s_{-2}$. The element $s$ in a Whittaker pair $(s, u)$ is called a {\it neutral element} for $u$ if there is a nilpotent element $v \in \mathfrak{g}$ such that $(v,s,u)$ is an $\mathfrak{sl}_2$-triple. A Whittaker pair $(s,u)$ with $s$ being a neutral element for $u$ is called a {\it neutral pair}.

Given any Whittaker pair $(s,u)$, define an anti-symmetric form $\omega_u$ on $\mathfrak{g}$ by
$$\omega_u(X,Y):=\kappa(u,[X,Y])\,.$$
For any $X \in \mathfrak{g}$, let $\mathfrak{g}_X$ be the centralizer of $X$ in $\mathfrak{g}$.
For any rational number $r \in \BQ$, let $\mathfrak{g}^s_{\geq r} = \oplus_{r' \geq r} \mathfrak{g}^s_{r'}$.
 Let $\mathfrak{u}_s= \mathfrak{g}^s_{\geq 1}$ and let $\mathfrak{n}_{s,u}$ be the radical of $\omega_u |_{\mathfrak{u}_s}$. Then it is clear that $[\mathfrak{u}_s, \mathfrak{u}_s] \subset \mathfrak{g}^s_{\geq 2} \subset \mathfrak{n}_{s,u}$. By \cite[Lemma 3.2.6]{GGS17}, $\mathfrak{n}_{s,u} = \mathfrak{g}^s_{\geq 2} + \mathfrak{g}^s_1 \cap \mathfrak{g}_u$.
Note that if the Whittaker pair $(s,u)$ comes from an $\mathfrak{sl}_2$-triple $(v,s,u)$, then $\mathfrak{n}_{s,u}=\mathfrak{g}^s_{\geq 2}$. Let $U_{s}=\exp(\mathfrak{u}_s)$ and $N_{s,u}=\exp(\mathfrak{n}_{s,u})$ be the corresponding unipotent subgroups of $\RG$. Define a character of $N_{s,u}$ by
$$\psi_u(n)=\psi(\kappa(u,\log(n)))\,.$$
 Let $N_{s,u}' = N_{s,u} \cap \ker (\psi_u)$. Then $U_s/N_{s,u}'$ is a Heisenberg group with center $N_{s,u}/N_{s,u}'$, here $U_s=\exp(\fu_s)$.

Let $\pi$ be an irreducible automorphic representation of $\RG(\BA)$. For any $\phi \in \pi$, the {\it degenerate Whittaker-Fourier coefficient} of $\phi$ attached to $(s,u)$ is defined to be
\begin{equation}\label{dwfc}
\CF_{s,u}(\phi)(g):=\int_{[N_{s,u}]} \phi(ng) \psi_u^{-1}(n) \, \mathrm{d}n\,.
\end{equation}
If $s$ is a neutral element for $u$, then $\CF_{s,u}(\phi)$ is also called a {\it generalized Whittaker-Fourier coefficient} of $\phi$.
Let $\CF_{s,u}(\pi)=\{\CF_{s,u}(\phi)|\phi \in \pi\}$.
The {\it wave-front set} $\mathfrak{n}(\pi)$ of $\pi$ is defined to the set of nilpotent orbits $\CO$ such that $\CF_{s,u}(\pi)$ is non-zero, for some neutral pair $(s,u)$ with $u \in \CO$. Note that if $\CF_{s,u}(\pi)$ is non-zero for some neutral pair $(s,u)$ with $u \in \CO$, then it is non-zero for any such neutral pair $(s,u)$, since the non-vanishing property of such generalized Whittaker-Fourier coefficients does not depend on the choices of representatives of $\CO$.
Let $\mathfrak{n}^m(\pi)$ be the set of maximal elements in $\mathfrak{n}(\pi)$ under the natural order of nilpotent orbits.
We recall a theorem from \cite{GGS17} in the following.

\begin{thm}[Theorem C, \cite{GGS17}]\label{ggsglobal1}
Let $\pi$ be an irreducible automorphic representation of $\RG(\BA)$.
Given a Whittaker pair $(s',u)$ and a neutral pair $(s,u)$, if $\CF_{s',u}(\pi)$ is non-zero, then $\CF_{s,u}(\pi)$ is non-zero.
\end{thm}

When $\RG$ is a quasi-split classical group, it is known that the nilpotent orbits are parametrized by pairs $(\underline{p}, \underline{q})$, where $\underline{p}$ is a partition and $\underline{q}$ is a set of non-degenerate quadratic forms (see \cite[Section I.6]{W01}).
When $\RG = \Sp_{2n}$, $\underline{p}$ is symplectic partition, namely, odd parts occur with even multiplicities.
When $\RG= \SO^{\alpha}_{2n}, \SO_{2n+1}$, $\underline{p}$ is orthogonal partition, namely, even parts occur with even multiplicities.
In these cases, let $\mathfrak{p}^m(\pi)$ be the set of partitions corresponding to nilpotent orbits in $\mathfrak{n}^m(\pi)$. A well-known folklore conjecture is that
$\mathfrak{p}^m(\pi)$ is a singleton. In this section,
for any symplectic or orthogonal partition $\underline{p}$, by a generalized Whittaker-Fourier coefficient of $\pi$ attached to $\underline{p}$, we mean a generalized
Whittaker-Fourier coefficient $\CF_{s,u}(\phi)$ attached to a nilpotent orbit $\CO$ parametrized by a pair $(\underline{p}, \underline{q})$ for some $\underline{q}$, with $\phi \in \pi$, $u\in \CO$ and $(s,u)$ being a neutral pair. For convenience, sometimes we also write a generalized Whittaker-Fourier coefficient attached to $\underline{p}$ as $\CF^{\psi_{\underline{p}}}(\phi)$,  without specifying the $F$-rational nilpotent orbit $\CO$ and neutral pairs.

For $\RG=\SO_{2n+1}$, an orthogonal partition $\underline{p}$ is called {\it special} if it has an even number of odd parts between two consecutive even parts and an odd number of odd parts greater than the largest even part (see \cite[Section 6.3]{CM93}). By the main results of \cite{JLS16}, any $\underline{p} \in \mathfrak{p}^m(\pi)$ is special. This will play an important role in the following.

\subsection{Non-vanishing of the descent}
Now we come back to the global situation where the groups and representations are the same as in \S \ref{section of residual representation}.
First we prove the following proposition.

\begin{prop}\label{toporbit}
$\CE_{\tau \otimes \sigma}$ has a non-zero generalized Whittaker-Fourier coefficient attached to the partition $[(2n)^2 1^3]$.
\end{prop}

\begin{proof}
Let $\alpha_i=e_i-e_{i+1}$ $(1 \leq i \leq 2n-1)$ be a subset of simple roots for $\SO^{V_0}_{4n+3}$. For $1 \leq i \leq 2n$, let $x_{\alpha_i}$ be the one-dimensional root subgroup in $\mathfrak{g}$ corresponding to $\alpha_i$.
By \cite[Section I.6]{W01}, there is only one nilpotent orbit $\CO$ corresponding to the partition $[(2n)^21^3]$. A representative of the nilpotent orbit $\CO$ can be taken to be 
$u = \sum_{i=1}^{2n-1} x_{-\alpha_i}(1)\,.$
Let $s$ be the following semi-simple element
$$s= \diag(2n-1, 2n-3, \ldots, 1-2n, 0, 0, 0, 2n-1, 2n-3, \ldots, 1-2n)\,.$$
Then it is clear that $(s,u)$ is a neutral pair.

We want to show that $\CF_{s,u}(\CE_{\tau \otimes \sigma})$ is non-zero.
To this end, we take another semisimple element
$$s'= \diag(4n, 4n-2, \ldots,  2, 0, 0, 0, -2, \ldots, -4n)\,.$$
It is clear that $(s',u)$ is a Whittaker pair. We consider $\CF_{s',u}(\CE_{\tau \otimes \sigma})$. Recall that $Q_{2n}$ is the parabolic subgroup of $\SO^{V_0}_{4n+3}$ with Levi subgroup isomorphic to $\GL_{2n} \times \SO_3(V_0)$ and unipotent radical subgroup $U_{2n}$. Then, by definition, for any $\phi \in \CE_{\tau \otimes \sigma}$, $\CF_{s',u}(\phi)$ is the constant term integral over $U_{2n}(F) \bs U_{2n}(\BA)$ combined with a non-degenerate Whittaker-Fourier coefficient of $\tau$.
Since $\CE_{\tau \otimes \sigma}$ is constructed from data $\tau \otimes \sigma$ on the Levi subgroup $\GL_{2n}(\BA) \times \SO_3(V_0)(\BA)$ with $\tau$ generic, the constant term integral over $U_{2n}(F) \bs U_{2n}(\BA)$ is non-zero and non-degenerate Whittaker-Fourier coefficients of $\tau$ are also non-zero. Hence, $\CF_{s',u}(\CE_{\tau \otimes \sigma})$ is non-zero.
Then, by Theorem \ref{ggsglobal1},
$\CF_{s,u}(\CE_{\tau \otimes \sigma})$ is also non-zero.
This completes the proof of the proposition.
\end{proof}

Next we prove the following.

\begin{prop}\label{toporbitraise}
$\CE_{\tau \otimes \sigma}$ has a non-zero generalized Whittaker-Fourier coefficient attached to the partition $[(2n+1)(2n-1)1^3]$.
\end{prop}

\begin{proof}
By Proposition \ref{toporbit}, we know that $\CE_{\tau \otimes \sigma}$ has a non-zero generalized Whittaker-Fourier coefficient attached to the partition $[(2n)^2 1^3]$. It is clear that as an orthogonal partition, $[(2n)^2 1^3]$ is not special, and the smallest special partition which is greater than it is $[(2n+1)(2n-1)1^3]$, which is called the special expansion of the partition
$[(2n)^2 1^3]$. By \cite[Theorem 11.2]{JLS16}, we must have that
$\CE_{\tau \otimes \sigma}$ has a non-zero generalized Whittaker-Fourier coefficient attached to the partition $[(2n+1)(2n-1)1^3]$.
\end{proof}

Now we are ready to prove Part (a) of Theorem \ref{main thm for descent}.
In the following, given $\beta \in F^{\times}$, we do not distinguish $\beta$ with its square class or the quadratic form corresponding to it. 

\begin{thm}\label{nonvanishing}
There exists $\beta \in F^\times$, such that
$\CD_{\psi_{n,\beta}}(\CE_{\tau \otimes \sigma})$ is non-zero.
\end{thm}

\begin{proof}
By Proposition \ref{toporbitraise}, we know that $\CE_{\tau \otimes \sigma}$ has a non-zero generalized Whittaker-Fourier coefficient attached to the partition $[(2n+1)(2n-1)1^3]$. By \cite[Section I.6]{W01}, nilpotent orbits corresponding to the partition $[(2n+1)(2n-1)1^3]$ are parametrized by certain quadratic forms $\{\beta_{2n+1}, \beta_{2n-1}, q_{V_0}\}$, corresponding to the parts $(2n+1), (2n-1)$ and $1^3$, where $\beta_{2n+1}$ and $\beta_{2n-1}$ are square classes, and $q_{V_0}$ is the quadratic form in $3$ variables on $V_0$ (see Section 2.1). This parametrization can be refined according to 
\cite[Proposition 8.1]{JLS16}, that is, $\CE_{\tau \otimes \sigma}$ actually has a non-zero generalized Whittaker-Fourier coefficient attached to the nilpotent $\CO$, corresponding to the partition $[(2n+1)(2n-1)1^3]$ and parametrized by quadratic forms $\{\beta,-\beta,q_{V_0}\}$ for some $\beta \in F^{\times}$.
Note that the normalization of the bilinear form for the irreducible representation of $\mathfrak{sl}_2(\mathbb{C})$ of dimension $i$ in \cite{JLS16} differs from the one in \cite{W01} by the factor $(-1)^{[(i-1)/2]}$ (See \cite[Section 7]{JLS16}).
In the following, we show that such a $\beta$ will suffice for the theorem.

For the nilpotent orbit $\CO$ above which is parametrized by quadratic forms $\{\beta, -\beta, q_{V_0}\}$, one can take a representative $u = u_1 + u_2$, where
$$u_1 = \sum_{i=1}^{n-1}x_{-\alpha_i}(1)+x_{e_{2n}-e_n}(1)+x_{-e_{2n}-e_n}(\beta/2)\,,$$
$u_2$ is the embedding into $\SO_{4n+3}^{V_0}$ of any representative
of a nilpotent orbit in the Levi part of the stabilizer of $u_1$ which is $\SO_{2n+2,\beta}$. This nilpotent orbit in $\SO_{2n+2,\beta}$ corresponds to the partition $[(2n-1)1^3]$ and is parametrized by
quadratic forms $\{-\beta, q_{V_0}\}$.
Let $s_1$ be the following semisimple element of the Lie algebra of $\SO_{4n+3}^{V_0}$:
$$\diag(2n, 2n-2, \ldots, 2, 0, \ldots, 0, -2, \ldots, 2-2n, -2n)\,.$$
Then it is clear that $(s_1, u_1)$ is a neutral pair.

We make $u_2$ explicit as follows.
First, we give a representative of the nilpotent orbit in $\SO_{2n+2,\beta}$ corresponding to the partition $[(2n-1)1^3]$ and parametrized by the
quadratic forms $\{-\beta, q_{V_0}\}$.
Let $s_2$ be the following semisimple element of the Lie algebra of $\SO_{2n+2, \beta}$:
$$\diag(2n-2, 2n-4, \ldots, 2, 0, 0,0,0,-2, \ldots, 4-2n, 2-2n)\,.$$
Let $N_{s_2}=\exp(\mathfrak{g}^{s_2}_{\geq 2})$, which is the unipotent radical of the parabolic subgroup of $\SO_{2n+2, \beta}$ with Levi subgroup isomorphic to $\GL_1^{n-1} \times \SO_{4, \beta}$, here $\SO_{4, \beta}=\SO(V_{0,\beta})$ (see \S \ref{subsection: construction of automorphic descent}). Write elements of $N_{s_2}$ as 
$n = \begin{pmatrix}
z & z x & y\\
0 & I_{4} & x'\\
0 & 0 & z^*
\end{pmatrix},$
where $z$ is an upper triangular matrix in $\GL_{n-1}$, $x \in \Mat_{(n-1) \times 4}$, and $x'$ is defined in Section \ref{subsection: construction of automorphic descent}. 
Let
$$u_2 = \sum_{i=1}^{n-2}x_{-\alpha_i}(1)+x_{e_{n}-e_{n-1}}(a)+x_{e_{n+1}-e_{n-1}}(b)+x_{-e_{n+1}-e_{n-1}}(c) + x_{-e_{n}-e_{n-1}}(d)\,,$$
 such that $(a,b,c,d) \in (F^{\times})^4$ is an anisotropic vector with respect the quadratic form of $\SO_{4,\beta}$. Then $(s_2,u_2)$ is a neutral pair and $u_2$ is a  representative of the nilpotent orbit in $\SO_{2n+2,\beta}$ corresponding to the partition $[(2n-1)1^3]$ and parametrized by
quadratic forms $\{-\beta, q_{V_0}\}$, and $N_{s_2}=N_{s_2,u_2}$ (see \S \ref{FCdefinition}). 
Then, we embed $s_2$ into the Lie algebra of $\SO_{4n+3}^{V_0}$ as follows:
$$\diag(0,\ldots, 0, 2n-2, 2n-4, \ldots, 2, 0, 0,0,0,0,-2, \ldots, 4-2n, 2-2n, 0, \ldots, 0)\,,$$
which is still denoted by $s_2$. We embed elements of $N_{s_2,u_2}$ into $\SO_{4n+3}^{V_0}$ as follows:
$$n = \begin{pmatrix}
z & z x & y\\
0 & I_{4} & x'\\
0 & 0 & z^*
\end{pmatrix} \mapsto \diag\left(I_n, \begin{pmatrix}
z & z \overline{x} & y\\
0 & I_{5} & \overline{x}'\\
0 & 0 & z^*
\end{pmatrix}, I_n\right),$$
where $\overline{x} = \{x_4, x^{(4)}, -\frac{2x_4}{\beta}\}$ when we write $x=\{x^{(4)}, 2x_4\}$ with $2x_4$ being the last column of $x$.
Still denote the image subgroup by $N_{s_2,u_2}$.
Similarly, we can embed $u_2$ into the Lie algebra of $\SO_{4n+3}^{V_0}$, and still denote the image by $u_2$.

Let $s=s_1 + s_2$. Then it is clear that $(s,u)$ is a neutral pair.
From the above discussion, there is a $\varphi \in \CE_{\tau \otimes \sigma}$ such that the generalized Whittaker-Fourier coefficient $\CF_{s,u}(\varphi) \neq 0$, for some choice of $u_2$. We fix such a choice.

Recall that
\begin{equation}\label{nvequ1}
\CF_{s,u}(\varphi)(g):=\int_{[N_{s,u}]} \varphi(ng) \psi_u^{-1}(n)\, \mathrm{d}n\,.
\end{equation}
Note that the elements of $N_{s,u}$ have the form:
$$n(x,y,z,w,\wt{n}):=\begin{pmatrix}
I_n & 0 & 0\\
w' & I_{2n+3} & 0\\
0 & w& I_n
\end{pmatrix}\begin{pmatrix}
z & z x & y\\
0 & \wt{n} &\wt{n} x'\\
0 & 0& z^*
\end{pmatrix},$$
where $z$ is an upper triangular matrix in $\GL_{n}$,
$\diag(I_n,\wt{n}, I_n)$ is in $N_{s_2,u_2}$, and $x \in \Mat_{n\times (2n+3)}$ and $w \in \Mat_{(2n+3)\times n}$ with some entries being zero.

To proceed, we define some unipotent subgroups. For $k=2, 3, \ldots, n-1$, let $R_k$ be the unipotent subgroup consists of elements of the form
$n=\begin{pmatrix}
I_n & r & 0\\
0 & I_{2n+3} & r'\\
0 & 0 & I_n
\end{pmatrix},$
such that $r_{i,j}$ are all zero except possibly when $i=k$ and
$1 \leq j \leq k-1$.
Similarly, for $k=2, 3, \ldots, n-1$, let $C_k$ be the unipotent subgroup consists of elements of the form
$n=\begin{pmatrix}
I_n & 0 & 0\\
c' & I_{2n+3} & 0\\
0 & c & I_n
\end{pmatrix},$
such that $c_{i,j}$ are all zero except possibly when $j=k+1$ and
$1 \leq i \leq k-1$.
One can see that $\prod_{k=1}^{n-1} R_k \cap N_{s,u} = \{I_{4n+3}\}$ and
$\prod_{k=1}^{n-1} C_k$ consists of all the elements in $N_{s,u}$ of the form $n(0,0,I_n,w,I_{2n+3})$.
Let $N_{s,u}'$ be the subgroup of $N_{s,u}$ consisting of all the elements of the form $n(x,y,z,0,\wt{n})$.

Now we apply \cite[Lemma 6.4]{JL16a} (It is clear that all the assumptions there hold,) to the quadruple
$$(N_{s,u}', \psi_u, \{R_k\}_{k=1}^{n-1}, \{C_k\}_{k=1}^{n-1})\,,$$
and obtain that the following integral
\begin{equation}\label{nvequ2}
\int_{[N'_{s,u}\prod_{k=1}^{n-1} R_k]} \varphi(ng) \psi_u^{-1}(n)\, \mathrm{d}n\neq 0,
\end{equation}
where elements in $N'_{s,u}\prod_{k=1}^{n-1} R_k$ have the form
$n(x,y,z,\wt{n}):=\begin{pmatrix}
z & z  x & y\\
0 & \wt{n} & \wt{n}  x'\\
0 & 0 & z^*
\end{pmatrix},$
here $z$ is an upper triangular matrix in $\GL_{n}$,
$\diag(I_n, \wt{n}, I_n)$ is in $N_{s_2,u_2}$, and $x \in \Mat_{n\times (2n+3)}$
with only $x_{n,j}=0$, for $1 \leq j \leq n-1$.

Then one finds that, as an inner integral of \eqref{nvequ2}, the integral
\begin{equation}\label{nvequ3}
\int_{[N_0]} \varphi(ng) \psi_{u_1}^{-1}(n)\, \mathrm{d}n
\end{equation}
is non-vanishing, where $N_0$ consists of all elements in
$N'_{s,u}\prod_{k=1}^{n-1} R_k$ having the form
$n(x,y,z,I_{2n+3})$. Note that here $x \in \Mat_{n \times (2n+3)}$
with only $x_{n,j}=0$, for $1 \leq j \leq n-1$.

Let $X$ be the unipotent subgroup consisting of elements of the form
$$n=\begin{pmatrix}≤
I_n & x & 0\\
0 & I_{2n+3} & x'\\
0 & 0 & I_n
\end{pmatrix},$$
such that $x_{i,j}$ are all zero except possibly when $i=n$ and
$1 \leq j \leq n-1$. Then $X$ is an abelian subgroup, preserving $N_0$ and the character $\psi_{u_1}$.
Taking Fourier expansion of the integral \eqref{nvequ3} along $[X]$, there exists a non-zero Fourier coefficient. Assume that one such non-zero Fourier coefficient is given by a Lie algebra element
$u_1'=\begin{pmatrix}
0_n & 0 & 0\\
y' & 0_{2n+3} & 0\\
0 & y & 0_n
\end{pmatrix}$, where $y'_{i,j}$ are all zero except possibly when $j=n$ and
$1 \leq i \leq n-1$.
Then we can rewrite this Fourier coefficient as
\begin{equation}\label{nvequ4}
\int_{[N_{s_1,u_1+u_1'}]} \varphi(ng) \psi_{u_1+u_1'}^{-1}(n) \, \mathrm{d}n\,.
\end{equation}

{We claim that $u_1$ and $u_1+u_1'$ are in the same nilpotent orbit, that is, one can find a element $g \in \SO_{4n+3}^{V_0}$ such that $u_1 = g (u_1+u_1') g^{-1}$. Indeed, write $u_1=\begin{pmatrix}
x & 0 & 0\\
y' & 0_{2n+3} & 0\\
0 & y & x^*
\end{pmatrix}$ and $u_1+u_1'=\begin{pmatrix}
w & 0 & 0\\
z' & 0_{2n+3} & 0\\
0 & z & w^*
\end{pmatrix}$, then $x=w \in \mathrm{Mat}_{(2n+3) \times n}$ and all columns of $y'$ and $z'$ are zero except the last columns $y'_n$ and $z'_n$. From the form of $u_1'$ above, regarding our choice of basis, one can see that as two vectors in an orthogonal space with action of $\SO_{2n+3}^{V_0}$, $y'_n$ and $z'_n$ have the same length, which equals $\beta\in F^\times$. Hence, one can find $h \in \SO_{2n+3}^{V_0}$ such that $y'_n = z'_n h^{-1}$. Let $g = \begin{pmatrix}
I_n & 0 & 0\\
0 & h & 0\\
0 & 0 & I_n
\end{pmatrix}$, then we have $u_1 = g (u_1+u_1') g^{-1}$.
Therefore, 
$u_1+u_1'$ is also a representative of the nilpotent orbit corresponding to the partition $[(2n+1)1^{2n+2}]$ and parametrized by quadratic forms $\{\beta, q\}$, where $q$ is a quadratic form in $(2n+2)$-variables such that $\beta \oplus q$ is isomorphic to $q_{V_0}$ composing with $n$ hyperplanes.
In other words, this nilpotent orbit is the same as
the one corresponding to the partition $[(2n+1)1^{2n+2}]$ and parametrized by quadratic forms $\{\beta, q_{V_0}\}$.}
It follows that $(s_1, u_1+u_1')$ is a neutral pair.
Therefore, we obtain that
the generalized Whittaker-Fourier coefficient $\CF_{s_1, u_1}$ is non-zero, that is,
$\CD_{\psi_{n,\beta}}(\CE_{\tau \otimes \sigma})$ is non-zero.

This completes the proof of the theorem.
\end{proof}

\begin{rmk}\label{remark for non-vanishing in general cases}
It is clear that the argument above can be easily modified and applied to the case that $\sigma$ is a cuspidal representation of $\SO(V_0)(\BA)$ with $\dim_F(V_0)=2r+1$, $r \in \BZ_{>0}$, and the residual representation $\CE_{\tau \otimes \sigma}$ for $\SO^{V_0}_{4n+2r+1}(\BA)$.
\end{rmk}

\section{On the reciprocal branching problem}\label{GGP}

In this section, we consider the reciprocal branching problem and prove Parts (b) and (c) of Theorem \ref{main thm for descent}.

For simplicity, write $\pi_\beta=\CD_{\psi_{n,\beta}}(\CE_{\tau\otimes\sigma})$.
By Theorem \ref{nonvanishing}, there exists $\beta\in F^\times$ such that $\pi_\beta\neq 0$.
Moreover, if the representation $\sigma$ satisfies Assumption \ref{depth assumption}, or the $\beta\in F^\times$ in Theorem \ref{nonvanishing} and $\omega_\tau$ satisfy the condition of Proposition \ref{prop for special beta}, then $\pi_\beta$ is cuspidal by Propositions \ref{prop for cuspidality}, \ref{prop for special beta}.
By the uniqueness of local Bessel models (see \cite{AGRS, GGP12, JSZ12, SZ}), one has a multiplicity free
direct sum decomposition:
$$\pi_{\beta}=\bigoplus_i \pi_\beta^{(i)}\,,$$
where $\pi_\beta^{(i)}$'s are non-zero mutual-inequivalent irreducible cuspidal automorphic representations of the group $G_{n+1,\beta}(\BA)=\SO_{2n+2,\beta}(\BA)$.
This proves Part (b) of Theorem \ref{main thm for descent}.

Let $\pi=\pi_\beta^{(i)}$ be any irreducible summand of $\pi_{\beta}$. By construction, there exists an automorphic form $\varphi_\pi\in \pi$ such that the inner product
\begin{equation}\label{zeta integral}
\pair{\varphi_\pi, \Res_{s=1/2}E^{\psi_{n,\beta}}(s, \cdot, \phi_{\tau\otimes\sigma})}\neq 0\,,
\end{equation}
for some choice of data. By Corollary 4.4 of \cite{JZ15}, the pair $(\pi,\sigma)$ has a non-zero Bessel period. 
In particular, we have

\begin{prop}\label{FCfordescent}
The automorphic descent $\pi_\beta=\CD_{\psi_{n,\beta}}(\CE_{\tau \otimes \sigma})$ has a non-zero Fourier coefficient attached to the nilpotent orbit corresponding to the partition $[(2n-1)1^3]$ and parameterized by quadratic forms $\{-\beta, q_{V_0}\}$.
\end{prop}

To get a connection with the reciprocal branching problem we have introduced in \S \ref{intro}, we need to study the Arthur parameter $\phi_\pi$ of each irreducible summand $\pi$ of $\pi_{\beta}$. In particular, we hope that $\phi_\pi$ is generic.
Combining with the local results we have obtained in \S \ref{local aspects}, we have the following proposition, which is Part (c) of Theorem \ref{main thm for descent}.

\begin{prop}\label{prop for parameter}
Assume that $\omega_\tau=1$, and the $\beta\in F^\times$ in Theorem \ref{nonvanishing} is not a square.
Then
each irreducible summand $\pi$ of the descent $\pi_\beta$
belongs to a global Arthur packet corresponding to a generic global Arthur parameter
$\phi_\pi$. The parameter $\phi_\pi$ has the central character $\eta_{V_{0,\beta}}$, and satisfies the property that
$L(\frac{1}{2}, \phi_\pi\times\phi_{\tau_0})\neq 0$ and $(\pi, \sigma)$ has a non-zero Bessel period.
\end{prop}

\begin{proof}
By Part (b) of Theorem 
\ref{main thm for descent}, under the assumption of the proposition, $\pi_\beta$ is cuspidal, so is $\pi$. 
Write $\pi=\otimes_v' \pi_v$, and consider all the finite places $v$ (infinitely many) such that both $\CE_{\tau \otimes \sigma,v}$ and $\psi_v$ are unramified and $\beta\in F_v^\times-{(F_v^\times)}^2$. Note that $\tau_v$ is a self-dual irreducible generic unitary unramified representation, which has the following form (see \cite{Ta86})
$$\tau_v = \chi_1 \lvert \cdot \rvert^{s_1} \times \cdots \times \chi_n \lvert \cdot \rvert^{s_n} \times  \chi_n^{-1} \lvert \cdot \rvert^{-s_n} \times \cdots \times \chi_1^{-1} \lvert \cdot \rvert^{-s_1}\,,$$
where $\chi_i$ is a unitary unramified character of $F_v^{\times}$ and $0 \leq s_i < \frac{1}{2}$, for each $1 \leq i \leq n$. 
By Proposition \ref{prop for unramified constituents I.1}, Part (1), 
$\pi_v$ is the fully-induced irreducible generic unitary unramified representation
$$\pi_v = \chi_1 \lvert \cdot \rvert^{s_1} \times \cdots \times \chi_n \lvert \cdot \rvert^{s_n} \rtimes 1\,.$$
By the endoscopic classification theory of Arthur \cite{A}, if $\pi$ has a non-generic global Arthur parameter, then $\pi_v$ is non-generic for almost all finite places.
Therefore, $\pi$ must have a generic global Arthur parameter $\phi_{\pi}$.
This proves the first statement.

The central character of the parameter $\phi_{\pi}$ is determined by the form of the group $G_{n+1,\beta}$, and hence is $\eta_{V_{0,\beta}}$.
From the discussion right before Proposition \ref{FCfordescent}, the pair $(\pi,\sigma)$ has a non-zero Bessel period, and hence by 
Theorem 5.7 of \cite{JZ15}, one has that 
$L(\frac{1}{2}, \phi_{\pi}\times\phi_{\tau_0})\neq 0$.
\end{proof}

To end this section, we state a theorem on the reciprocal branching problem, which completes the proof of Theorem \ref{main thm for descent}.

\begin{thm}\label{main thm for reciprocal problem}
Let $\tau=\tau_1 \boxplus \cdots \boxplus \tau_r$ be an irreducible isobaric sum automorphic representation of $\GL_{2n}(\BA)$ with trivial central character, such that each $\tau_i$ is a unitary irreducible cuspidal representation of $\GL_{n_i}(\BA)$ of orthogonal type.
Let $V_0$ be a quadratic space of dimension $3$ over $F$, and $\sigma$ be an irreducible cuspidal automorphic representation of $\SO(V_0)(\BA)$ which lies in the global Vogan packet $\wt\Pi_{\phi_{\tau_0}}[H_1^*]$ (here $H_1^*=\SO_3^*$, $F$-split) for some irreducible cuspidal automorphic representation $\tau_0$ of $\GL_2(\BA)$ of symplectic type.
Assume that
\begin{itemize}
\item[(i)] $L(\frac{1}{2},\tau\times\tau_0)\neq 0$,
\item[(ii)] the representation $\sigma$ satisfies Assumption \ref{depth assumption}.
\end{itemize}
Then there exists a $\beta\in F^\times$, such that the twisted descent $\pi_\beta=\CD_{\psi_{n,\beta}}(\CE_{\tau\otimes\sigma})$
has all of its irreducible summands $\pi_\beta^{(i)}$, as cuspidal automorphic representations of $G_{n+1,\beta}(\BA)$, enjoying the property that
each $(\pi_\beta^{(i)},\sigma)$ has a non-trivial Bessel period.
Moreover, if the $\beta\in F^\times$ taken in Theorem \ref{nonvanishing} is not a square, then each $\pi_\beta^{(i)}$ has a generic global Arthur parameter $\phi^{(i)}$, and gives an answer to the reciprocal branching problem with respect to $\sigma\in \wt\Pi_{\phi_{\tau_0}}[H_1^*]$.
\end{thm}

\begin{rmk}\label{rmk for global parameters}
A more precise description of those parameters $\phi^{(i)}$ can be deduced from the refined local theory of the global zeta integrals in \cite{JSdZ}, we will not discuss them here. 
\end{rmk}

\end{document}